\documentclass[11pt,a4paper]{article}
\usepackage[OT1]{fontenc}
\usepackage[english]{babel}
\usepackage{amsfonts}
\usepackage{amsthm}
\def\h{ {\cal H} }
\def\a{ {\cal A} }

\def\l{ {\cal L} }
\def\n{ {\cal N} }

\def\b{ {\cal B} }
\def\u{ {\cal U} }
\def\m{ {\cal M} }
\def\ii{ {\cal I} }

\def\s{ {\cal S} }

\def\p{ {\cal P} }

\def\f{ {\cal F} }

\def\d{ {\cal D} }
\def\j{ {\cal J} }
\def\hh{ \mathbb{H}_{(P_0,Q_0)}}

\newtheorem{teo}{Theorem}[section]
\newtheorem{prop}[teo]{Proposition}
\newtheorem{lem}[teo]{Lemma}
\newtheorem{coro}[teo]{Corollary}
\newtheorem{defi}[teo]{Definition}
\theoremstyle{definition}
\newtheorem{rem}[teo]{Remark}
\newtheorem{ejem}[teo]{Example}

\title{Projections with fixed difference: a Hopf-Rinow theorem.}
\author{Esteban Andruchow, Gustavo Corach, L\'azaro Recht}

\begin{document}

\maketitle 

\begin{abstract}
The set $\d_{A_0}$, of pairs of orthogonal  projections $(P,Q)$ in generic position with fixed difference $P-Q=A_0$, is shown to be a homogeneus smooth manifold: it is the quotient of the unitary group of the commutant $\{A_0\}'$ divided by the unitary subgroup of the  commutant $\{P_0, Q_0\}'$, where $(P_0,Q_0)$ is any fixed pair in $\d_{A_0}$. Endowed  with a natural reductive structure (a linear connection) and the quotient Finsler metric of the operator norm, it behaves as a classic Riemannian space: any two pairs in $\d_{A_0}$  are joined by a  geodesic of minimal length. Given a base pair $(P_0,Q_0)$,  pairs in an open dense subset of $\d_{A_0}$ can be joined to $(P_0,Q_0)$ by a {\it unique} minimal geodesic. 
\end{abstract}

\bigskip

{\bf 2010 MSC:}  58B20, 47B15.

{\bf Keywords:}  Projections, pairs of projections, differences of projections.

\section{Introduction}
Let $\h$ be a Hilbert space, denote by $\b(\h)$ the algebra of bounded linear operators in $\h$, and by  $\p(\h)$ the set of (orthogonal) projections in $\h$.  We study here the class $\d$  of operators which are differences of projections,
$$
\d=\{ P-Q: P,Q\in\p(\h)\},
$$
and for  $A\in\d$,  the set
$$
\d_A=\{ (P,Q)\in\p(\h)\times\p(\h): P-Q=A\}.
$$
Results on differences $P-Q$ appeared since the 1940's, as part of the {\it two subspaces problem}: to find a complete set of unitary invariants for a pair of closed subspaces $M,N$ (or equivalently, for a pair of projections $P,Q$). This problem was solved by J. Dixmier \cite{dixmier}, who obtained a characterization of $\d$. An operator $A\in\b(\h)$ belongs to $\d$ if and only if $A^*=A$, $\|A\|\le 1$  and there exists a symmetry $V$ in $\h'=N(A^2-1)^\perp$ such that $AV=VA$ in $\h'$ (a {\it symmetry} is a selfadjoint unitary operator). This form of Dixmier's result is due to Ch. Davis \cite{davis}, who found a nice  solution of  the two subspaces problem by a geometric study of the closeness and separation  operators of a pair $P,Q$: $C(P,Q)=PQP+(1-P)(1-Q)(1-P)$ is called the closeness operator of $P, Q$, and $S(P,Q)=P(1-Q)P+(1-P)Q(1-P)$ is the separation operator of $P, Q$. Observe that, if $A=P-Q$, then $C=1-A^2$ and $S=A^2$.
In \cite{ass} J. Avron, R. Seiler and B. Simon defined and studied Fredholm pairs of projections, and an index for them: $(P,Q)$ is a Fredholm pair if $P|_{R(Q)}:R(Q)\to R(P)$ is a Fredholm operator, whose index is called the index of the pair. Their methods rely on an extensive use  of the differences $A=P-Q$ and $B=P+Q-1$. For a nice presentation of these results, see W. Amrein and K. Sinha \cite{amreinsinha} 

A more recent study of $\d$ can be found in \cite{pemenoscu},  where several known facts on the differential geometry of $\p(\h)$ were used to describe, for instance, the interior and boundary of $\d$, its connected components, and also some special parts of $\d$ (elements in $\d$ which are Fredholm, compact,  or nuclear). 

In \cite{chinos}, W. Shi., G. Ji and H. Du studied several properties of $\d_A$, for any $A\in\d$. In particular, they proved that $\d_{A_0}\subset\b(\h_0)$ is connected, where $\h_0=\{N(A^2-1)\oplus N(A)\}^\perp$ and $A_0=A|_{\h_0}$. 

The main goal of this paper is to present $\d_{A_0}$ as a homogeneous space and a differentiable manifold. 
As such, following ideas of Dur\'an, Mata-Lorenzo and Recht \cite{coco}, the space $\d_{A_0}$  has a natural invariant Finsler metric. Also, using a well known characterization by Halmos  \cite{halmos}, of pairs of projections in generic position, we show that $\d_{A_0}$ has a reductive structure, a fact which enables one to introduce a linear connection in this space, and to compute its geodesics (given by one-parameter unitary groups acting on a given pair $(P_0, Q_0)$). We show that with the Finsler metric and the reductive structure,  $\d_{A_0}$ satisfies a Hopf-Rinow theorem: pairs in $\d_{A_0}$  are joined by a  geodesic of minimal length. Moreover, on a dense open subset of $\d_{A_0}$, such geodesic is unique.

In Section 2, we present Davis' characterization of $\d$ by means of the {\it Halmos decomposition} of $\h$ (in the presence of a pair $P,Q\in\p(\h)$). Using Davis' and Halmos' tools, we show that the Friedrich's angle is constant in $\d_A$. Recall (see Deutsch \cite{deutsch}) that $\alpha_F(M,N)\in [0,\pi/2]$ is the Friedrich's angle between the closed subspaces $M,N$ if
$$
\cos(\alpha_F(M,N))=\sup\{ |\langle \mu,\nu\rangle| : \mu\in M\ominus N, \nu\in N\ominus M, \|\mu\|=\|\nu\|=1\}=\|P_MP_N-P_{M\cap N}\|.
$$ 
 Moreover, it is shown that $\cos(\alpha_F(M,N))=\|P_0Q_0\|$ (=constant) for any $P=P_M$, $Q=P_N$ such that $P-Q=A$, where $P_0$, $Q_0$ denote the reductions of $P$, $Q$ to the common invariant subspace $\h_0=\{N(A)\oplus N(A^2-1)\}^\perp$.  Hereafter, $P_0, Q_0, A_0$ will be called the {\it generic part} of $P,Q,A$, respectively. Also in this section we show that, with the usual order of positive definite operators, the set $\{P_0+Q_0: (P_0,Q_0)\in\d_{A_0}\}$ cannot be ordered: $P_0+Q_p\le P'_0+Q'_0$ if and only if $P_0=P'_0$ and $Q_0=Q'_0$.
In Section 3 we introduce the action of the unitary group $\u_\a$ of 
$$
\a=\{A\}'=\{T\in\b(\h): TA=AT\}
$$
on $\d_A$. $\a$ and $\u_\a$ are also reduced by $\h_0$.  It is proven that the generic part of $\u_\a$ acts transitively on $\d_{A_0}$, and from this follows that $\d_{A_0}$ is connected (as proved by Shi, Ji and Du in \cite{chinos}). Section 4 contains a description of $\a_0$ in terms of Halmos' decomposition, which will be used later. In Section 5 we present some examples in $\d$. In Section 6 we show that $\d_{A_0}$ is a differentiable homogeneous manifold, with a natural reductive structure. For instance, the geodesic curves can be computed, and we show that the exponential map of the linear connection is surjective. We endow the tangent spaces of $\d_{A_0}$ with the quotient norm, as defined by Dur\'an, Mata-Lorenzo and Recht in \cite{coco}, and show that with this metric, the geodesics of the reducive connection are minimal up to the border of $\d_{A_0}$.

\section{Davis' characterization}

If $T\in\b(\h)$, denote by $R(T)$ and $N(T)$  the range and nullspace of $T$, respectively. If $A\in\d$, the space $\h$ can be decomposed orthogonally as
\begin{equation}\label{3space}
\h=N(A)\oplus N(A^2-1)\oplus \h_0,
\end{equation}
where $\h_0=\left(N(A)\oplus N(A^2-1)\right)^\perp$.
Note that $N(A^2-1)=N(A-1)\oplus N(A+1)$. 
For any presentation $A=P-Q$, it is straightforward to verify that
$$
N(A)=R(P)\cap R(Q)\oplus N(P)\cap N(Q) \ , \ \ N(A-1)=R(P)\cap N(Q) \hbox{ and } N(A+1)=N(P)\cap R(Q).
$$
So that the decomposition (\ref{3space}) is essentially the decomposition considered by Dixmier \cite{dixmier} and Halmos \cite{halmos} to study the equivalence of pair of projections. In particular, the subspace $\h_0$ is usually called the {\it generic part} of $P$ and  $Q$, or more properly, the generic part of $A=P-Q$. Therefore, the decomposition (\ref{3space}) reduces simultaneously  any pair $P$, $Q$ in $\d_A$.

Using the decomposition (\ref{3space}), the set $\d_A$ is factorized as follows:
\begin{enumerate}
\item
In the subspace $N(A^2-1)=N(A-1)\oplus N(A+1)$, $A$ is given by
$A=1_{N(A-1)}\oplus -1_{N(A+1)}$. That is, any pair $(P,Q)\in\d_A$ coincides with $(P_{N(A-1)},P_{N(A+1)})$  in this subspace. 
\item
In the subspace $N(A)$, the pairs $(P,Q)\in\d_A$ reduce to pairs of the form $(P',P')$, with $P'\in\p(N(A))$. Thus, if $N(A)$ is non trivial,  the structure of $\d_A|_{N(A)}$ is that of $\p(N(A))$.
\item
The structure of $\d_A$ in $\h_0$ was characterized by Davis \cite{davis}.  In Theorem \ref{davis} below we describe the results obtained by Davis \cite{davis} on this set.
\end{enumerate}

A {\it symmetry} $V\in\b(\h)$ is a selfadjoint unitary operator: $V^*=V^{-1}=V$. Symmetries are special cases of difference of projections: $V=P_{+1}-P_{-1}$, where $P_{\pm 1}$ are the orthogonal projections onto the eigenspaces $\{\xi\in\h: V\xi=\pm\xi\}$. Also note that $V=2P_{+1}-1$ and $P_{\pm 1}=\frac12(1\pm V )$.

Let us summarize the information above:

\begin{rem}
In the decomposition $\h=N(A)\oplus N(A^2-1)\oplus \h_0$, the set $\d_A$ is decomposed as
$$
\d_A=\p_{N(A)}\oplus \{A_{\pm 1} \} \oplus \d_{A_0},
$$
where $A_{\pm 1}=P_{N(A-1)}-P_{N(A+1)}$ is a symmetry. It follows that $\d_A$ consists of a single element if and only if $A$ is a symmetry.
\end{rem}
As announced, let us describe the structure of $\d_{A_0}$:
\begin{teo}\label{davis} {\rm (essentially \cite{davis})}
Let $A\in\d$, and let $A_0$ be its generic part. There exist one to one correspondences between
\begin{itemize}
\item
Pairs $(P_0,Q_0)$ such that $P_0-Q_0=A_0$.
\item
Symmetries $V$ in $\h_0$ such that $VA_0=-A_0V$.
\item
Closed subspaces $\s$ of $\h_0$ such that $A_0(\s)\subset \s^\perp$ and $A_0(\s^\perp)\subset \s$.
\item
Projections $E\in\p(\h_0)$ such that $EA_0E=(1-E)A_0(1-E)=0$.
\end{itemize}
\end{teo}
\begin{proof}
Given a pair $(P_0,Q_0)\in\d_{A_0}$, one obtains a symmetry which anti-commutes with $A_0$ as follows. Consider the selfadjoint operator $S=P_0+Q_0-1$. Note that $S=P_0-(1-Q_0)$ is also a difference of projections. Its nullspace is trivial:
$$
N(S)=R(P_0)\cap R(1-Q_0) \oplus N(P_0)\cap N(1-Q_0)=R(P_0)\cap N(Q_0)\oplus N(P_0)\cap R(Q_0)=\{0\}.
$$
Therefore, the polar decomposition of $S$, $S=V |S|=|S|V$ yields a symmetry $V$ (note that $V$ is the sign function of $S$). Clearly, $SP_0=P_0Q_0=Q_0S$ and $SQ_0=Q_0P_0=P_0S$. In particular, this implies that $S^2$ commutes  with $P_0$ and $Q_0$. Then $|S|=(S^2)^{1/2}$ also commutes with both projections. It follows that 
$V P_0=Q_0V$ and $V Q_0 =P_0V$. Then
$$
VA_0=VP_0-VQ_0=Q_0V-P_0V=-A_0V.
$$

Given a symmetry $V$ which anti-commutes with $A_0$, put  (see \cite{davis}, p. {\bf 181})
$$
P_V=\frac12\{ 1+A_0+(1-A_0^2)^{1/2}V\} \ \hbox{ and } \ \ Q_V=\frac12\{ 1-A_0+(1-A_0^2)^{1/2}V\}.
$$
Straightforward computations show that $P_V, Q_V\in\p(\h_0)$,  $P_V-Q_V=A_0$, and $P_V+Q_V-1=(1-A_0^2)^{1/2}V$. Then, since $V$ and $A_0^2$ commute,
$$
 ( P_V+Q_V-1)^2=1-A_0^2 ,  \hbox{ i.e., } |P_V+Q_V-1|=(1-A_0^2)^{1/2},
$$  
and $P_V+Q_V-1=|P_V+Q_V-1|V$. That is, the correspondence betwen pairs and symmetries is reciprocal. 

Given a symmetry $V$ which anti-commutes with $A_0$, let $\s=\{\xi\in\h_0: V\xi=\xi\}$, so that $\s^\perp=\{\xi\in\h_0: V\xi=-\xi\}$. If $\xi\in\s$, $VA_0\xi=-A_0V\xi=-A_0\xi$, i.e., $A\xi\in\s^\perp$. Similarly, $A(\s^\perp)\subset \s$. The converse holds: if $A_0(\s)\subset S^\perp$ and $A(\s^\perp)\subset \s$, then the symmetry $V=P_\s-P{\s^\perp}=2P_\s-1$ anti-commutes with $A_0$. In fact,
if $\xi\in\s$, 
$$
(2P_\s-1)A_0\xi=2P_\s A_0\xi-A_0\xi=-A_0\xi=-A_0(2P_\s-1)\xi;
$$
if $\eta\in\s^\perp$,
$$
(2P_\s-1)A_0\eta=2P_\s A_0\eta-A_0\eta=2A_0\eta-A_0\eta =A_0\eta=-A_0(2P_\s-1)\eta.
$$

Given a closed subspace $\s\subset \h_0$ such that $A_0(\s)\subset\s^\perp$ and $A_0(\s^\perp)\subset\s$, the orthogonal projection $E=P_\s$ satisfies that $EA_0E=(1-E)A_0(1-E)=0$, and conversely. 
\end{proof}
\begin{rem} \label{anexo davis}
Since $A_0$ is selfadjoint with trivial nullspace, the isometric part $J_0$, in the polar decomposition $A_0=J_0|A_0^2|=|A_0|J_0$, is a symmetry. Note that a symmetry $V$ anti-commutes with $A_0$, if and only if it anti-commutes with $J_0$. Indeed, $V$ commutes with $A_0^2$ and with $|A_0|$, which has also trivial nullspace:
$$
|A_0|J_0V=-V|A_0|J_0=-|A_0|VJ_0,
$$
which implies that $J_0V=-VJ_0$. The converse is trivial. Therefore, the following equivalent condition could be added in the previous theorem:
\begin{itemize}
\item
Symmetries $V$ in $\h_0$ such that $VJ_0=-J_0V$.
\end{itemize}
\end{rem}

\begin{rem}
Note that $A\in\d$ is a selfadjoint contraction. Moreover, the existence of a symmetry intertwining $A_0$ with $-A_0$, means that the spectrum of the whole $A$ is symmetric with respect to the origin, except for an eventual asymmetry at $\lambda=\pm 1$. For instance, if $0<\lambda<1$ is an eigenvalue of $A$, then $\pm\lambda\in\sigma(A_0)$,  with the same multiplicity. The symmetry may break at $\lambda=1$.
\end{rem}

\begin{rem}\label{no hay par menor}
Consider now the following question: among the pairs $(P,Q)\in\d_A$, does the exist an optimal element which minimizes $P+Q$?
We use the decomposition $\h=N(A)\oplus N(A^2-1)\oplus \h_0$ which reduces all  pairs in $\d_A$. In the first subspace $N(A)$, all pairs are of the form $(E,E)$, for $E$ a projection onto a subspace of $N(A)$. Clearly, there is a minimal pair here, taking $E=0$. On $N(A^2-1)$, there is one pair, and for this pair $P+Q$ equals the identity of $N(A^2-1)$. Let us prove that pairs $(P_0,Q_0)$  in the generic part $\h_0$ are not comparable (unless they are equal). This implies that in the nonntrivial case, where the generic part $\h_0\ne 0$, there are no possible minimizers for $P+Q$. For $(P_0, Q_0)\in\d_{A_0}$, let $V_0$ be the corresponding Davis symmetry: $P_{V_0}=P_0$, $Q_{V_0}=Q_0$. Then
$$
P_0+Q_0=1+V_0(1-A_0^2)^{1/2}.
$$
Thus, comparison of these sums is equivalent to comparison  of the operators $V_0(1-A_0^2)^{1/2}=(1-A_0^2)^{1/2}V_0$. If $V_1$ is the symmetry corresponding to another pair $(P_1,Q_1)$, then, since
$$
\langle V_0(1-A_0^2)^{1/2}\xi,\xi\rangle=\langle V_0(1-A_0^2)^{1/4}\xi, (1-A_0^2)^{1/4}\xi\rangle,
$$
it follows that $P_1+Q_1\le P_0+Q_0$ if and only if
$$
\langle V_0(1-A_0^2)^{1/4}\xi, (1-A_0^2)^{1/4}\xi\rangle \le \langle V_1(1-A_0^2)^{1/4}\xi, (1-A_0^2)^{1/4}\xi\rangle .
$$
Moreover, since $1-A_0^2$ has trivial nullspace, $(1-A_0^2)^{1/4}$ has dense range. Thus, the inequality above is equivalent to 
$V_1\le V_0$.  This inequality is equivalent, in turn, to the inclusion $\s_1^+\subset \s_0^+$, where $S_i^+=\{\xi\in\h_0: V_i\xi=\xi\}$. Therefore our assumption $P_1+Q_1< P_0+Q_0$ implies the existence of a nontrivial vector $\xi_0\in\s_0^+$ such that $\xi_0\perp\s_1^+$. This leads us to a contradiction. In fact, note  that $A_0V_0=-V_0A_0$ implies that $A_0(\s_i^+)\subset (\s_i^+)^\perp$ and $A_0((\s_i^+)^\perp)\subset \s_i^+$. Then 
$$
A_0\xi_0\in(\s_0^+)^\perp \ \hbox{ and } \ \ A_0\xi_0\in A_0((\s_1^+)^\perp)\in \s_0^+,
$$
i.e., $A_0\xi_0=0$, a contradiction, since $A_0$ has trivial nullspace. 
\end{rem}

\subsection{Halmos decomposition}
Given two projections $P,Q$, Halmos proved in \cite{halmos} that there exists an isometric isomorphism between the generic part $\h_0$ (of $P$ and $Q$) and a product space $\l\times \l$ which carries $P_0$ and $Q_0$ to the operator matrices
$$
\left( \begin{array}{cc} 1 & 0 \\ 0 & 0 \end{array}\right) \ \hbox{ and } \ \   \left( \begin{array}{cc} C^2 & CS \\ CS & S^2 \end{array}\right),
$$
respectively. Here $C=\cos(\Gamma)$ and $S=\sin(\Gamma)$, where $0\le\Gamma\le \pi/2$ is a positive operator  in $\l$ with trivial nullspace. In particular, $CS=SC$ and $S$ has trivial nullspace. Note that also $C$ has trivial nullspace (i.e., $\pi/2$ is not an eigenvalue of $X$). Indeed, if $C\xi=0$, then 
$$
\left( \begin{array}{cc} 1 & 0 \\ 0 & 0 \end{array}\right)\left(\begin{array}{c} \xi \\ 0 \end{array} \right)=\left(\begin{array}{c} \xi \\ 0 \end{array} \right) \  \hbox{ and } \ \left( \begin{array}{cc} C^2 & CS \\ CS & S^2 \end{array}\right)\left(\begin{array}{c} \xi \\ 0 \end{array} \right)=\left(\begin{array}{c} 0 \\ 0 \end{array} \right),
$$
i.e., $\left(\begin{array}{c} \xi \\ 0 \end{array} \right)$ lies in $R(P_0)\cap N(Q_0)$, which is trivial.

\subsection{Friedrichs' angle}

Given two closed subspaces $\m,\n\subset \h$, the Friedrichs angle  \cite{friedrichs} between $\m$ and $\n$ is the angle whose cosine is  
$$
c(\m,\n):=\sup\{|\langle\mu,\nu\rangle|: \mu \in \m\ominus\n, \nu \in \n\ominus\m,  \|\mu\|=\|\nu\|=1\}. 
$$
It holds that $c(\m,\n)=\|P_\m P_\n - P_{\m\cap\n})\|$ (see \cite{deutsch}).

We prove next that if $(P,Q)\in\d_A$, then $c(R(P),R(Q))$ does not depend on the pair, i.e., it  is an invariant of $A$.
\begin{prop}
Friedrichs' angle $c(R(P),R(Q))=c(N(P),N(Q))$ is constant for all $(P,Q)$ in $\d_A$.
\end{prop}
\begin{proof}
Pick $(P',Q'), (P,Q)\in\d_A$. Let us reduce $P'Q'-P_{R(P')\cap R(Q')}$ and $PQ-P_{R(P)\cap R(Q)}$ in the three space decomposition (\ref{3space}). Note that $R(P)\cap R(Q)$ and $R(P')\cap R(Q')$  are non trivial only in $N(A)$. In $N(A)$ and $N(A^2-1)=R(P)\cap N(Q)\oplus N(P)\cap R(Q)$, $PQ-P_{R(P)\cap R(Q)}$ is trivial, and similarly for $(P',Q')$.
In the generic part $\h_0$, by Theorem \ref{transitiva}, there exists a unitary operator $U$ such that $UP_0U^*=P'_0$ and $UQ_0U^*=Q'_0$. Then  $U(R(P_0))=R(P'_0)$ and $U(R(Q_0))=R(Q'_0)$, so that 
$$
U(R(P_0)\cap R(Q_0))=R(P'_0)\cap R(Q'_0) \  , \ \hbox{  i.e.,  } UP_{R(P_0)\cap R(Q_0)}U^*=P_{R(P'_0)\cap R(Q'_0)}.
$$
Then,  in the three space decomposition (\ref{3space})
$$
P'Q'-P_{R(P')\cap R(Q')}=0 \oplus 0 \oplus U\left(P_0Q_0-P_{R(P_0)\cap R(Q_0)}\right)U^*,
$$
and, thus,
$$
\|P'Q'-P_{R(P')\cap R(Q')}\|=\|P_0Q_0-P_{R(P_0)\cap R(Q_0)}\|=\|PQ-P_{R(P)\cap R(Q)}\|,
$$
i.e., $c(R(P'),R(Q'))=c(R(P),R(Q))$.
\end{proof}
\begin{rem}
Note that  $c(R(P),R(Q))=\|P_0Q_0\|$. Using Halmos representation
$$
\|P_0Q_0\|=\|P_0Q_0P_0\|^{1/2}=\|\left(\begin{array}{cc} C^2 & 0 \\ 0 & 0 \end{array} \right)\|^{1/2}=\|C\|.
$$
Then, the angle equals $\cos^{-1}(\|\cos(\Gamma)\|)$. If $\Gamma$ is non invertible, $0\in\sigma(\Gamma)$ and therefore 
$1\in\sigma(C)$, and thus the angle is $\pi/2$. If $\Gamma$ is invertible (which is equivalent to $A^2-1$ being of closed range), then $\|\cos(\Gamma)\|=\cos(\|\Gamma^{-1}\|^{-1})$, and the angle is $\|\Gamma^{-1}\|^{-1}$, or, equivalently, the lowest value in the spectrum of $\Gamma$. In any case, this quantity is an invariant of $A$. We shall see below (Remark \ref{norma suma}), that $\|P_0 + Q_0\|$ is also an invariant of $A_0$.
\end{rem}

\section{A unitary action on $\d_A$}

Let $\a:=\{A\}'=\{T\in\b(\h): TA=AT\}$. Since $A$ is selfadjoint, $\a$ is a von Neumann subalgebra of $\b(\h)$.
Let $\u_\a$ be the unitary group of $\a$. Observe that $\u_\a$ is connected.  The group $\u_\a$ acts on $\d_A$:
$$
U\cdot (P,Q)=(UPU^*,UQU^*) \ , \ U\in\u_\a \ , \ (P,Q)\in\d_A,
$$
because $UPU^*-UQU^*=UAU^*=A$.

The algebra $\a$ and the action of $\u_\a$ can be factored using the three space decomposition (\ref{3space}) $\h=N(A)\oplus N(A^2-1)\oplus \h_0$. The algebra $\a$ splits as
$$
\a=\b(N(A)) \oplus \a_1 \oplus \a_0.
$$
Let us describe the summands. The first summand is clearly $\b(N(A))$: any operator acting on $N(A)$ (and trivial in $N(A)^\perp$) commutes with $A$ ($=0$ in $N(A)$).
A pair $(P,Q)\in\d_A$ reduces to $(P',P')$ in $N(A)$ , for some projection $P' \in  \b(N(A))$. The action of the unitary group of $N(A)$ on these pairs is essentially the action of the unitary group of a space on the projections of the space. The orbits are parametrized by the dimensions of the range and the nullspace.

The second summand consists of the algebra of operators which commute with $A|_{N(A^2-1)}$. In the decomposition $N(A^2-1)=N(A-1)\oplus N(A+1)$, $A|_{N(A^2-1)}$ is the matrix
$$
A|_{N(A^2-1)}=\left( \begin{array}{cc} 1 & 0 \\ 0 & -1 \end{array} \right)
$$
and  the operators in $N(A^2-1)$ which commute with $A$ are of the form
$$
T=\left( \begin{array}{cc} T_1 & 0 \\ 0 & T_2 \end{array} \right).
$$
The unitary operators of this form leave $A|_{N(A^2-1)}$ fixed (a fact consistent with the observation that all pairs $(P,Q)\in\d_A$ reduce to a unique element in $N(A^2-1)$).

The third summand is $\a_0:=\{A|_{\h_0}: A\in\a\}$.
Therefore, it is natural to focus on the action of $\u_{\a_0}$, the unitary group of the part $\a_0$.

\begin{teo} \label{transitiva}
The action of $\u_{\a_0}$ on $\d_{A_0}$ is transitive.
\end{teo}
\begin{proof}
Let $(P_0,Q_0), (P'_0,Q'_0)\in\d_{A_0}$. Denote by $V$ and $V'$ the symmetries (which anti-commute with $A_0$) which correspond to these pairs. Consider the decomposition
$$
\h_0=N(P_0+Q'_0-1)\oplus N(P_0+Q'_0-1)^\perp.
$$
Note that $N(P_0+Q'_0-1)$ reduces simultaneously both pairs $(P_0,Q_0)$, $(P'_0,Q'_0)$. First note that  $N(P_0+Q'_0-1)=N(P'_0+Q_0-1)$, because $P_0-Q_0=P'_0-Q'_0$. Also note that 
$$
N(P_0+Q'_0-1)=N(P_0-(1-Q'_0))= R(P_0)\cap R(1-Q'_0) \oplus N(P_0)\cap N(1-Q'_0)
$$
$$
=R(P_0)\cap N(Q'_0) \oplus N(P_0)\cap R(Q'_0),
$$
which reduces $P_0$ and $Q'_0$, and similarly for $P'_0$ and $Q_0$. 

In the second subspace $N(P_0+Q'_0-1)^\perp$, the operator $P_0+Q'_0-1$ is selfadjoint and has trivial nullspace; therefore, in the polar decomposition
$$
P_0+Q'_0-1=\Sigma |P_0+Q'_0-1|=|P_0+Q'_0-1|\Sigma,
$$
the operator $\Sigma$ is a symmetry which, by the same argument as in the proof of Theorem \ref{davis}, satifies
$$
\Sigma P_0|_{N(P_0+Q'_0-1)^\perp}\Sigma=Q'_0|_{N(P_0+Q'_0-1)^\perp}\ \  \hbox{ and } \ \  \Sigma Q'_0|_{N(P_0+Q'_0-1)^\perp}\Sigma=P_0|_{N(P_0+Q'_0-1)^\perp}.
$$
Then, 
$$
\Sigma A_0|_{N(P_0+Q'_0-1)^\perp} \Sigma =-A_0|_{N(P_0+Q'_0-1)^\perp}.
$$
The fact that $P_0+Q'_0-1=P'_0+Q_0-1$, implies that this  operator $\Sigma$ intertwines also the reductions of $P'_0$ and $Q_0$ to $N(P_0+Q'_0-1)^\perp$

The symmetry $V$, which is obtained (by means of the Borel functional calculus) as the sign function of $P_0+Q_0-1$, also is reduced by $N(P_0+Q'_0-1)$. Clearly, $V|_{N(P_0+Q'_0-1)^\perp}$ also anti-commutes with $A_0|_{N(P_0+Q'_0-1)^\perp}$.
Then, the unitary operator $U_1$ in $N(P_0+Q'_0-1)^\perp$ defined as
$$
U_1=\Sigma V|_{N(P_0+Q'_0-1)^\perp},
$$
commutes with $A_0|_{N(P_0+Q'_0-1)^\perp}$. Moreover, it satisfies
$$
U_1P_0|_{N(P_0+Q'_0-1)^\perp}U_1^*=\Sigma (VP_0V)|_{N(P_0+Q'_0-1)^\perp}\Sigma=\Sigma Q_0|_{N(P_0+Q'_0-1)^\perp}\Sigma=P'_0|_{N(P_0+Q'_0-1)^\perp}
$$
and, similarly,
$$
U_1Q_0|_{N(P_0+Q'_0-1)^\perp}U_1^*=Q'_0|_{N(P_0+Q'_0-1)^\perp}.
$$   

Let us find a unitary in the other subspace, $N(P_0+Q'_0-1)$. Trivially, $P_0|_{N(P_0+Q'_0-1)}=1-Q'_0|_{N(P_0+Q'_0-1)}$, and also
$P'_0|_{N(P_0+Q'_0-1)}=1-Q'_0|_{N(P_0+Q'_0-1)}$.
Note  that 
$$
A_0 P_0|_{N(P_0+Q'_0-1)}=(1-Q_0)P_0|_{N(P_0+Q'_0-1)}=P'_0(1-Q'_0)|_{N(P_0+Q'_0-1)}=P'_0A_0|_{N(P_0+Q'_0-1)},
$$
and, similarly,
$$
A_0 Q_0|_{N(P_0+Q'_0-1)}=Q'_0A_0|_{N(P_0+Q'_0-1)}.
$$
Then, again by the same argument as above (and as in the proof of Theorem \ref{davis}), the isometric  part $U_2$ in the polar decomposition of $A_0|_{N(P_0+Q'_0-1)}$ (which has trivial nullspace in the whole $\h_0$ and, thus also in $N(P_0+Q'_0-1)$), is a symmetry ($U_2=U_2^*$) which satisfies
$$
U_2P_0|_{N(P_0+Q'_0-1)}U_2=P'_0|_{N(P_0+Q'_0-1)}  \ \hbox{ and } \  U_2Q_0|_{N(P_0+Q'_0-1)}U_2=Q'_0|_{N(P_0+Q'_0-1)}
$$
In particular, this implies that $U_2$ commutes with $A_0|_{N(P_0+Q'_0-1)}$.
Consider, then,
$$ 
U=U_2\oplus U_1 \ \hbox{ acting in }  \ N(P_0+Q'_0-1)\oplus N(P_0+Q'_0-1)^\perp=\h_0.
$$
Clearly, $U$ is a unitary operator which commutes with $A_0$, and satisfies $U\cdot(P_0,Q_0)=(P'_0,Q'_0)$.
\end{proof}
The following result appeared in \cite{chinos}. It was proved there using a different technique: Shi, Ji and Du obtained a parametrization of $\d_{A_0}$, in terms of unitaries in a von Neumann algebra. The result  is proved here as an easy consequence of the above theorem:
\begin{coro}
$\d_{A_0}$ is connected. The connected components of $\d_A$ are parametrized by the connected components of the space of projections $\p(N(A))$ of the space $N(A)$.
\end{coro}
\begin{proof}
$\u_{A_0}$ is the unitary group of a von Neumann algebra, therefore connected, and the map $\pi_{(P_0,Q_0)}$ is continuous. The assertion on the components of $\d_A$ follows form the description of $\d_A$ done in Section 1.
\end{proof} 
\begin{rem}
The unitary $U$ obtained above is, in fact, an explicit formula in  terms of $P_0,Q_0,P'_0$ and $Q'_0$. However, if one fixes for instance the pair $(P_0,Q_0)$, $U$ is not a continuous formula in terms of $(P'_0,Q'_0)$ (a continuous formula would provide a continuous global cross-section for the action). Indeed, the formula of $U$ depends on  the decomposition 
$\h_0=N(P_0+Q'_0-1)\oplus N(P_0+Q'_0-1)^\perp$. Or, equivalently, on  the map 
$$
Q'_0\mapsto P_{N(P_0+Q'_0-1)}=P_{N(P_0-(1-Q'_0)}.
$$
One can find trivial examples (in dimension $2$, for instance) where this map is not continuous.
\end{rem} 
However, in some cases  the action does have continuous {\it local} cross  sections. Let us show one such case.

Given a fixed $(P_0,Q_0)\in\d_{A_0}$, consider the continuous (surjective) map
$$
\pi_{(P_0,Q_0)}:\u_{\a_0}\to \d_{A_0} \ , \ \ \pi_{(P_0,Q_0)}(U)=U\cdot(P_0,Q_0)=(UP_0U^*,UQ_0U^*).
$$
\begin{lem} \label{cerrado}
$1-A^2$ has closed range if and only if   for any $(P_0,Q_0)\in\d_{\a_0}$,  $P_0+Q_0-1$ is invertible. 
\end{lem}
\begin{proof}
Suppose that $1-A^2$ has closed range. Note the formula (see \cite{kato} p. 33, or compute directly):
$$
(P-Q)^2+(P+Q-1)^2=1,
$$
or, equivalently, $1-A^2=(P+Q-1)^2$. It follows that $(P+Q-1)^2$ has closed range. In the generic part $\h_0$, 
$(P+Q-1)^2|_{\h_0}=(P_0+Q_0-1)^2$ has trivial nullspace. Thus, $(P_0+Q_0-1)^2$ is invertible, and, thus, also $P_0+Q_0-1$ is invertible.

Conversely, if $P_0+Q_0-1$ is invertible, then $(P_0+Q_0-1)^2$ is also invertible, and  then its extension   $(P+Q-1)^2=1-A^2$ (which is zero in $N(P+Q-1)$) has closed   range.
\end{proof}

\begin{rem}
Using Halmos decomposition, a simple computation shows  that $R(A^2-1)$ is closed, which means that $A_0^2-1$ is invertible if and only if  $S$ (or $\Gamma$)  is invertible in $\l$.
\end{rem}

For such $A$ as above, the map $\pi_{(P_0,Q_0)}$ has continuous {\it local } cross-sections.

\begin{prop}
Let $A\in\d$ such that $A^2-1$ has closed range. Then the map $\pi_{(P_0,Q_0)}$ has continuous local cross-sections.

\end{prop}
\begin{proof}
Consider the set 
$$
\{(P',Q')\in\d_{A_0}: P_0+Q'-1  \hbox{ is invertible in }\h_0\}.
$$
Since the set of invertible operators is open, this set is clearly an open subset of $\d_{A_0}$ (considered with the relative topology of  $\b(\h_0)\times\b(\h_0)$). It is a neighbourhood of $(P_0,Q_0)$: if $(P',Q')=(P_0,Q_0)$, $P_0+Q_0-1$ is invertible, by the above Lemma. Then the map
$$
{\bf s}: \{(P',Q')\in\d_{A_0}: P_0+Q'-1  \hbox{ is invertible in }\h_0\}\to \u_{\a_0} , \ {\bf s}(P',Q')=sgn(P_0+Q'-1)V
$$
is continuous. Here $sgn(P_0+Q'-1)$ denotes the sign of the selfadjoint (invertible) operator $P_0+Q'-1)$, $sgn(P_0+Q'-1)=(P_0+Q'-1)\left((P_0+Q'-1)^2\right)^{-1/2}$. Note that the function $sgn$ is continuous on the set of invertible operators.  As seen above, it is an element of $\u_{\a_0}$ (called $\Sigma V$ in the proof of Theorem \ref{transitiva}). Also, it is clear that it is a cross section  in a neighbouthood of $(P_0,Q_0)$. One obtains cross sections around other points by translating this map using the transitive action.
\end{proof}

\begin{rem}\label{norma suma}
In Remark \ref{no hay par menor}, we observed that if  $(P_0,Q_0),(P'_0,Q'_0)$ are the  generic parts of  two pairs $(P,Q), (P',Q')\in\d_A$, the operators $P_0+Q_0$ and $P_0'+Q'_0$ are not comparable (an inequality implies equality). The transitivity of the action of $\u_{\a_0}$ on $\d_{A_0}$ implies that the  norms of these operators coincide. Indeed, since there exists a unitary operator $U$ in $\a_0$ such that $UP_0U^*=P'_0$ and $UQ_0U^*=Q'_0$, it follows that $U(P_0+Q_0)U^*=P'_0+Q'_0$, and therefore $\|P_0+Q_0\|=\|P'_0+Q'_0\|$. 
\end{rem}

\section{A presentation of $\a_0$ in terms of Halmos decomposition}

\begin{prop}
The algebra $\a_0$, represented in $\b(\l\times\l)$, consists of matrices of the form
$$
\{\left( \begin{array}{cc} X & Y \\ Y & Z \end{array}\right) : X,Y,Z\in\b(\l) \hbox{ commute with } \Gamma , \hbox{ and }  C(X-Z)+2SY=0 \}.
$$
\end{prop}
\begin{proof}
In $\l \times \l$,  $A_0$ is 
$$
\left( \begin{array}{cc} S^2 & -CS \\ -CS & -S^2 \end{array}\right).
$$
Let  $\left( \begin{array}{cc} X & Y_1 \\ Y_2 & Z \end{array}\right)$ be an operator which  commutes with $A_0$. Then, in particular, it commutes with $A_0^2$ which is given by
$$
\left( \begin{array}{cc} S^2 & 0 \\ 0 & S^2 \end{array}\right) .
$$
Then $X,Y_1,Y_2,Z$ commute with $S^2$. Therefore,  they commute also with its square root $|S|=S$, and  with $C$. Thus, $X,Y_1,Y_2,Z$ commute with $e^{i\Gamma}=C+iS$, and with its analytic logarithm $i\Gamma$ (since $\|\Gamma\|\le \pi/2<\pi$).
Straightforward computations  show that an operator lies  in the commutant of $A_0$ if and only if 
\begin{itemize}
\item 
 $CSY_1=CSY_2$,  which means that $Y_1=Y_2$, because   $C, S$  have  trivial nullspaces, and
\item
$S^2Y-CSZ=-CSX-S^2Y$, which, again using that $S$ has trivial nullspace, means that  $C(X-Z)+2SY=0$.
\end{itemize}
\end{proof}

\begin{rem}
Note that since $S=\sin(\Gamma)$ has trivial nullspace, then the (eventually unbounded, densely defined) operator $\tau=\tan(\Gamma)$ is defined, and the condition $C(X-Z)+2SY=0$ can be replaced by
$$
Z=X+2\tau Y.
$$
In particular, it implies that $\tau Y=Y\tau$ is bounded.
\end{rem}
\begin{rem}\label{isotropia}
It is also easy to characterize the unitaries in $\a_0$ which leave $(P_0,Q_0)$ fixed. They are the (unitary) matrices which commute  with $P_0$ and $Q_0$. The first relation implies that they must be diagonal matrices. Commutation with the second projection implies, after simple computations (using that $C$ has trivial nullspace),  that they are of the form
$$
\{W\in\u_{\a_0}:WP_0W^*=P_0 \hbox{ and } WQ_0W^*\}=\{ \left( \begin{array}{cc} W'  & 0 \\ 0 & W' \end{array}\right) : W'\in\u(\l), W'\Gamma=\Gamma W'\}.
$$
\end{rem}
Using this representation of the generic part, we can further analize the condition that $A^2-1$ has closed range. Recall from Lemma \ref{cerrado}, that this is equivalent to the invertibility of $P_0+Q_0-1$, for any pair $(P_0,Q_0)\in\d_{A_0}$.
\begin{prop}
The following are equivalent
\begin{enumerate}
\item
$A^2-1$ has closed range.
\item
$PQP-P$ has closed range.
\item
$P_0Q_0P_0-P_0$ is invertible in $R(P_0)$
\item
$\Gamma$ is invertible in $\l$.
\end{enumerate}
\end{prop}
\begin{proof}
Clearly, it suffices to examine the reductions to the generic part $\h_0$. Using Halmos representation, one gets
$$
A_0^2-1=\left(\begin{array}{cc} -S^2 & 0 \\  0 &  -S^2 \end{array} \right),
$$
and  $P_0Q_0P_0-P_0 =-S^2=-\sin(\Gamma)$. The equivalence of these conditions becomes apparent.
\end{proof}
\section{Examples}
We  present  examples of operators $A$, which will be the object of further study. The first one has continuous spectrum.
\begin{ejem}\label{mt}
Let $\h=L^2(-1,1)$ and $A=M_t$ (multiplication by the variable): $Af(t)=tf(t)$. Note that $A$ anti-commutes with the symmetry $V$, $Vf(t)=f(-t)$. Therefore $A=P_V-Q_V$, following the notation of Davis' characterization in Section 1, and both projections can be computed explicitly. Since $A$ has no eigenvalues, it follows that $\h_0=\h$ (i.e., $P_V,Q_V$ or any pair of projections with difference $A$ are in generic position). Also note that the algebra $\a$ is $L^\infty(-1,1)$, represented as multiplication operators in $\h$.  Therefore, if one chooses to parametrize elements  in $\d_A$ by means of isometries, $\d_A$ consists of all symmetries $V_\varphi$ of the form
$$
V_\varphi=M_\varphi V M_{\bar{\varphi}},
$$
for $\varphi\in L^\infty(-1,1)$, with $|\varphi(t)|=1 \ a.e.$, modulo the commutant of $V$, i.e., the unimodular functions of $L^\infty(-1,1)$ which are {\it essentially even}. Explicitely,
$$
V_\varphi f(t)=\varphi(t)\bar{\varphi}(-t)f(t),
$$
modulo the functions $\varphi$ such that $\varphi(t)=\varphi(-t)$ for almost every $t\in(-1,1)$.
\end{ejem}
The second example has pure point spectrum.
\begin{ejem}\label{piqj}
Let $I,J\subset\mathbb{R}^n$ be Lebesgue measurable sets with positive finite measure. Consider $\h=L^2(\mathbb{R}^n)$ and the projections 
$$
P_I=M_{\chi_I} \ \hbox{ and } Q_J=\f^{-1}P_J\f
$$
where $\chi_D$ is the characteristic function of $D\subset\mathbb{R}^n$ and $\f:\h\to\h$ is the Fourier-Plancherel transform. These pairs have been studied in connection with mathematical formulation of the uncertainty principle (see \cite{lenard}, the survey \cite{folland} or the book \cite{havin}). Specifically, the products $P_IQ_J$, $P_IQ_JP_I$ are of interest. Among the basic facts concerning these operators, it is known that they are Hilbert-Schmidt operators, and that
$$
R(P_I)\cap R(Q_J)=R(P_I)\cap N(Q_J)=N(P_I)\cap R(Q_J)=\{0\},
$$
and $N(P_I)\cap N(Q_J)$ is infinite dimensional (see for instance  \cite{lenard}). In particular, $P_IQ_JP_I$ has a complete orthonormal system of eigenvectors (i.e., $P_IQ_JP_I$ is diagonalizable). In \cite{pqvsp-q} it was proved that for a pair of projections $P,Q$, $PQP$ is diagonalizable if and only if $P-Q$ is diagonalizable.  In this case, there is an explicit relation between the eigenvectors and eigenvalues of $PQP$ and $P-Q$. 
 If  $s_n$ are the eigenvalues of $PQP$ ($0<s_n<1$), then  $\pm\lambda_n=\pm(1-s_n)^{1/2}$ are the eigenvalues of $P-Q$. The eigenvalue $s=1$ corresponds with $\lambda=0$.

Thus, our second example $A=P_I-Q_J$ is diagonalizable.  Moreover, 
$$
\h_0=N(A)^\perp=\left(N(P_I)\cap N(Q_J)\right)^\perp.
$$

In the particular case $I=(0,1)$ and $J=(-\Omega/2,\Omega/2)$ the eigenvectors are known (called prolate spheroidal functions \cite{slepian}, \cite{havin}), and the eigenvalues have simple multiplicity. Therefore, in this case $\a_0$ consists of all diagonal matrices in this orthonormal basis. In particular, $\a_0$ is commutative, as in the previous example.

Let us characterize  in this example the symmetries which anti-commute with $A_0$. Note that this implies that $V$ commutes with $A_0^2$. Therefore $V$ has block diagonal form, with blocks of size $2\times 2$, generated, for each fixed $n\ge 1$,  by   the eigenvectors $e_n, f_n$ of $\lambda_n$ and $-\lambda_n$, respectively. Note that $Ve_n$ is an eigenvector for $-\lambda_n$:
$$
A_0Ve_n=-VA_0e_n=-\lambda_nVe_n.
$$
Thus, since  in this case all eigenvalues have multiplicity one,  $Ve_n=\omega_nf_n$, for some $\omega_n\in\mathbb{C}$ with $|\omega_n|=1$ . Similarly, $Vf_n$ is an unimodular multiple of $e_n$. The fact that $V^2=1$ implies that $Vf_n=\bar{\omega}_ne_n$. Therefore, any symmetry $V$ anti-commuting with $A_0$ (in $\h_0$) is of the form
$$
V=V_\omega=\oplus_{n=1}^\infty V_{\omega_n}\in \oplus_{n=1}^\infty \h_n \ ,  \ \ \hbox{ where } \ \  V_{\omega_n}=\left( \begin{array}{cc} 0 & \omega_n \\ \bar{\omega}_n & 0 \end{array} \right) 
$$
and $\h_n$ is the subspace spanned by $e_n$ and $f_n$. That is, the elements of $\d_{A_0}$  can be parametrized by sequences $\omega=\{\omega_n\}$ of complex numbers of modulus $1$.
\end{ejem}
\begin{rem}\label{entradas conmutan}
If we consider the $2\times 2$ matrix representation of the previous section, elementary computations show that if $\a_0$ is commutative, then all entries in the matrices commute.
\end{rem}
\begin{ejem}
Consider $\h=L^2(\mathbb{T})$,  put $\h^+=H^2(\mathbb{T})$, $P_+=P_{\h^+}$. Let ${\bf a}=\{a_1,\dots, a_N\}$, ${\bf b}=\{b_1,\dots,b_N\}$ two (finite) sequences  of points in the open disk $\mathbb{D}$. We suppose that $a_i\ne a_j$ if $i\ne j$ and $a_i\ne b_j$ for all $i,j$. Let $B_{\bf a}$, $B_{\bf b}$ be the corresponding Blaschke products. Put $P_{\bf a}=P_{B_{\bf a}\h^+}$, and similarly $P_{\bf b}=P_{B_{\bf b}\h^+}$. Note that since the multiplication operator $M_{B_{\bf a}}$ is a unitary operator in $\h$, $P_{\bf a}=M_{B_{\bf a}} P_+ M_{\bar{B}_{\bf a}}$ and similarly  $P_{\bf b}=M_{B_{\bf b}} P_+ M_{\bar{B}_{\bf b}}$. Also note that, if ${\bf a}\# {\bf b}=\{a_1,\dots, a_N,b_1,\dots, b_N\}$, then
$$
\h_{\bf a}\cap \h_{\bf b}=\h_{{\bf a}\# {\bf b}} \ ;
$$
$$
\h_{\bf a}\cap \h_{\bf b}^\perp = \h_{\bf a}^\perp \cap \h_{\bf b}=\{0\}.
$$
This is an exercise which follows from the fact  that ${\bf a}$ and ${\bf b}$ have the same cardinality (see also \cite{grassH2} for a proof). Moreover,
$$
\h_{\bf a}^\perp\cap \h_{\bf b}^\perp=(\h_{\bf a}\cup  \h_{\bf b})^\perp=(\h^+)^\perp=\h^-
$$
because $B_{\bf a}$ and $B_{\bf b}$ are coprime inner functions ($a_i\ne b_j$). Therefore, the generic part $\h_0$ equals the model space
$$
\h_0=\h^+\ominus \h_{{\bf a}\# {\bf b}},
$$
spanned by $\{k_{a_1},\dots, k_{a_N},k_{b_1},\dots, k_{b_N}\}$, where $k_c$ denotes the reproducing kernel  of $\h^+$ at $c\in\mathbb{D}$. Since we have chosen different points $a_i$, $b_j$ in the disk, these functions are linearly independent, and $\h_0$ has dimension $2N$. Thus $\a_0$ is a finite dimensional algebra. The operator $A_0=P_{\bf a}-P_{\bf b}|_{\h_0}$ is a $2N\times 2N$ matrix. Let us compute $A_0$ in the elements of the basis $k_{a_i}, k_{b_j}$. First, note that $k_{a_i}\in \h_a^\perp$: if $h\in\h^+$, $\langle B_{\bf a}h, k_{a_i}\rangle= B_{\bf a}(a_i)h(a_i)=0$. Then
$$
P_{\bf a} k_{a_i}=0 \ , \ \ P_{\bf b} k_{b_j}=0.
$$
Note that 
$$
P_+\bar{B}_{\bf a} k_{b_j}=\sum_{l\ge 0} \langle \bar{B}_{\bf a} k_{b_j}, z^l\rangle z^l=\sum_{l\ge 0}\langle k_{b_j}, B_{\bf a} z^l\rangle z^l=\sum_{l\ge 0} \overline{B_{\bf a}(b_j )} (\bar{b}_jz)^l= \overline{B_{\bf a}(b_j )}\frac{1}{1-\bar{b}_jz}=\overline{B_{\bf a}(b_j )}k_{b_j}.
$$
Then, $P_ak_{b_j}=M_{B_{\bf a}}P_+M_{\bar{B}_{\bf a}}k_{b_j}=\overline{B_{\bf a}(b_j )}B_{\bf a}k_{b_j}$. Similarly for $P_{\bf b}k_{a_i}$. Then,
$$
A_0k_{a_i}= -\overline{B_{\bf b}(a_i )}B_{\bf b}k_{a_i}
$$
and 
$$
A_0k_{b_j}= \overline{B_{\bf a}(b_j )}B_{\bf a}k_{b_j}.
$$
Thus, in principle, it is possible to compute the $2N\times 2N$ matrix of $A_0$ in the (non-orthogonal) basis of the reproducing kernels $k_{a_i}, k_{b_j}$. It would be interesting to know if under the present assumptions, $A_0$ has eigenvalues of simple multiplicity.
\end{ejem}

\begin{ejem}\label{idempotente}
Let $E\in\b(\h)$ be a non-selfadjoint  idempotent operator ($E^2=E$). Consider the orthogonal projections $P_{R(E)}$ and $P_{R(E^*)}=P_{N(E)^\perp}$. In matrix form, in terms of the decomposition $\h=R(E)\oplus R(E)^\perp$, $E$ is written
$$
E=\left( \begin{array}{cc} 1 & B \\ 0 & 0 \end{array} \right),
$$
where $B: R(E)^\perp \to R(E)$. 
Consider the selfadjoint operator $S=E+E^*-1$. $S$ is selfadjoint with trivial nullspace, which satisfies $SE=E^*S$ and $SE^*=ES$. Then, similarly as before, the unitary part of $S$ in the polar decomposition intertwines $E$ and $E^*$. Then, it also intertwines the range projections. Moreover, by straightforward matrix computations (which were done explicitly in \cite{timisoara}), this unitary part coincides with Davis' symmetry  $V$ for the pair of projections $P_{R(E)}, P_{R(E^*)}$.
Also note the well known formulas
$$
P_{R(E)}=ES^{-1} \ \hbox{ and } \ P_{R(E^*)}=E^*S^{-1},
$$
so that 
\begin{equation}\label{AdeB}
A=(E-E^*)S^{-1}.
\end{equation} 
Note that $R(E)\cap N(E)=\{0\}$ and $R(E^*)\cap N(E^*)=\{0\}$. 
Straightforward computation show that 
$$
R(E)\cap R(E^*)=N(B^*)\  \hbox{ and }\  N(E^*)  \cap N(E)=N(B).
$$ 
Thus, in order that $P_{R(E)}$ and $P_{R(E^*)}$ be in generic position,  $B$  should have dense range and trivial nullspace. Let us assume this. In particular it implies that $\dim R(E)=\dim R(E)^\perp$. 
Clearly, if we want to study the structure of $\d_A$, we can replace $E$ with $UEU^*$, where $U:\h\to \j$ is a unitary transformation.  Thus, $A=P_{R(E)}-P_{R(E^*)}$ is replaced by $UAU^*$. Pairs in $\d_A$ are mapped onto pairs in $\d_{UAU^*}$ by means of $(P,Q)\mapsto (UPU^*,UQU^*)$. Therefore  (by the equality of dimensions between $R(E)$ and $R(E)^\perp$) we can choose a model $\j=\l\times\l$. Next consider the polar decomposition of $B$, $B=W_0|B|$. Clearly $W_0:\l\to\l$ is a unitary operator. Consider the unitary operator $W$ in $\l\times\l$ given by $W=\left(\begin{array}{cc} W_0^* & 0 \\ 0 & 1 \end{array}\right)$. Then,
$$
WEW^*=\left(\begin{array}{cc} 1 & W_0^*B \\ 0 & 0 \end{array}\right)=\left(\begin{array}{cc} 1 & |B| \\ 0 & 0 \end{array}\right).
$$
Summarizing, we can suppose that $\h=\l\times\l$ and $B$ is positive with trivial nullspace. Therefore, with the current assumptions, using (\ref{AdeB}), one has
$$
A_0=A=\left(\begin{array}{cc} B^2(1+B^2)^{-1} & -B(1+B^2)^{-1} \\  -B(1+B^2)^{-1} & -B^2(1+B^2)^{-1}\end{array} \right) .
$$
In order to describe $\a_0=\{A_0\}'$, note that $A_0^2=\left(\begin{array}{cc} B^2(1+B^2)^{-1} & 0 \\  0 & B^2(1+B^2)^{-1}\end{array} \right)$. Therefore, if $X=\left(\begin{array}{cc} X_{11} & X_{12} \\  X_{21} & X_{22} \end{array} \right)$ belongs to $\a_0$, in particular it commutes with $A_0^2$. This clearly implies that the entries $X_{ij}$ commute with $B$ (recall that $B\ge 0$). Next,  note that the condition that $X$ commutes with $A_0$ means that
$$
BX_{12}=X_{21}B=BX_{21},
$$
which implies $X_{12}=X_{21}$, because $N(B)=0$, and that
$$
BX_{22}-X_{11}B=B(X_{22}-X_{11})=2B^2X_{12}, 
$$
which implies that $X_{22}=2BX_{12}+X_{11}$. Therefore 
$$
\a_0=\{ \left( \begin{array}{cc} Y & Z \\ Z & Y+2BZ \end{array} \right): Y,Z \hbox{ commute with } B\}.
$$
Let us describe the isotropy subalgebra, i.e., the operators which commute with $P_{R(E)}$ and $P_{R(E^*)}$. Operators which commute with $P_{R(E)}=\left( \begin{array}{cc} 1 & 0 \\ 0 & 0 \end{array} \right)$, are diagonal matrices. If they belong additionally to $\a_0$, they are of the form 
$$
\{ \left( \begin{array}{cc} Y & 0 \\ 0 & Y \end{array} \right): Y \hbox{ commutes with } B\}.
$$
Easy examples (of positive operators $B$) , show that the isotropy subalgebra, and therefore $\a_0$, may not be commutative.
\end{ejem}
\section{A regular structure for $D_A$}
We shall prove that  $\d_A$ is an homogeneous $C^\infty$ space of the unitary group $\u_{\a}$. If  $A^2-1$ has closed range, then $\d_A$ is additionally a complemented submanifold of $\p(\h)\times \p(\h)$.
As seen above, $\d_A$ decomposes as three spaces in the decomposition (\ref{3space}). 
\begin{itemize}
\item
In $N(A)$, the group acting is the whole unitary group of $N(A)$, and the space $\d_A$ reduces to pairs of the form $(E,E)$, where $E\in\p(\h)$ and  $R(E)\subset N(A)$, i.e., $\d_A|_{N(A)}$ identifies with the space of projections in the Hilbert space $N(A)$, under the action of the unitary group of $N(A)$. This space is well studied: it is a $C^\infty$ complemented submanifold of $\b(N(A))$ (see \cite{cpr}). 
\item
In $N(A^2-1)$, $\d_A|_{N(A^2-1)}$ is a single point , namely, $(P_{N(A-1} , P_{N(A+1)})$.
\item
Therefore, the task is reduced to show that $\d_{A_0}$ has local regular structure.
\end{itemize}

In order to prove that $\d_{A_0}$ has differentiable structure, and also in order to define later a linear connection in this manifold, the following map will be useful:
\begin{defi}\label{esperanza}
Fix $(P_0,Q_0)\in\d_{A_0}$, and fix also a Halmos decomposition for this pair. 
Consider the map $E_{(P_0,Q_0)}: \a_0\to \a_0$,
 \begin{equation}\label{formula esperanza}
 E_{(P_0,Q_0)}(\left( \begin{array}{cc} X & Y \\ Y & Z \end{array}\right))= \left( \begin{array}{cc} \frac12(X+Z) & 0 \\ 0 & \frac12(X+Z) \end{array}\right).
\end{equation}
\end{defi}
It is easy to see that this map is a conditional expectation, with range equal to the subalgebra of elements in $\a_0$ which commute with $P_0$ and $Q_0$ (see Remark \ref{isotropia}). 

\begin{rem}\label{independencia de Halmos}
Let us prove that the conditional expectation $E_{(P_0,Q_0)}$ depends only on the pair $(P_0,Q_0)$ (and not on the Halmos decomposition). Indeed, first note that $Q_0-P_0Q_0P_0-P_0^\perp Q_0P_0^\perp=\left(\begin{array}{cc} 0 & CS \\ CS & 0 \end{array}\right)$. Let us denote this operator by $K$. Since $C$ and $S$ have trivial nullspace, then $N(K)=0$. Also, it is clear that $K^*=K$. Then, in the polar decomposition of $K$,
$$
K=W|K|=W \left( \begin{array}{cc} CS & 0 \\ 0 & CS \end{array}\right);
$$
using again that $C$ and $S$ have trivial nullspaces, it follows that it must be $W=\left( \begin{array}{cc} 0 & 1 \\ 1 & 0 \end{array}\right)$. The operator $W$ is obtained by means of the functional calculus of $K$: $W=sgn(K)$, where $sgn$ denotes the (Borel, eventually non continuous)  sign function ($sgn(t)=1$ if $t\ge 0$, $-1$ if $t<0$). Note that $P_0W=\left( \begin{array}{cc} 0 & 1 \\ 0 & 0 \end{array}\right)$ and $WP_0=\left( \begin{array}{cc} 0 & 0 \\ 1 & 0 \end{array}\right)$. Then, if $M=\left( \begin{array}{cc} X & Y \\ Y & Z \end{array}\right)$, we get
$$
E_{(P_0,Q_0)}(M)=\frac12 P_0(M+WMW)P_0+\frac12 P_0^\perp(M+WMW)P_0^\perp .
$$
\end{rem}

\begin{defi}\label{el horizontal}
As above, fix $(P_0,Q_0)\in\d_{A_0}$ and a Halmos decomposition for this pair. Denote
 \begin{equation}\label{reductivo}
\mathbb{H}_{(P_0.Q_0)}:=N(E_{(P_0,Q_0)}\cap (\a_0)_{ah}=\{\left(\begin{array}{cc} -Y\tau & Y \\ Y & Y\tau \end{array} \right): Y^*=-Y,  Y\Gamma=\Gamma Y, Y\tau \hbox{ is bounded in } \l\}.
\end{equation}
Note that  if $Y\tau$ is bounded, then  $Z-X=2Y\tau$. 
\end{defi}
\begin{rem}\label{CONMUTAN}
Let us point out the  fact, which will be relevant later, that for any (matrix) element in $\mathbb{H}_{(P_0,Q_0)}$, all entries of the matrix commute. Indeed,  $Y$ commutes with $\tau=\tan(\Gamma)$.
\end{rem}

In order to study  the local structure of $\d_{A_0}$ we first suppose that $R(A^2-1)$ is closed. In this case we shall prove that $\d_{A_0}$ is a submanifold of $\b(\h_0)\times\b(\h_0)$ (as well a a homogeneous space of $\u_{\a_0}$). To prove this fact we shall need the following lemma, which is an application of the inverse function theorem in Banach spaces. One can find a detailed and elementary  proof of this fact in \cite{rae}.

\begin{lem}\label{raeburn}
Let $G$ be a Banach-Lie group acting smoothly on a Banach space $X$. For a fixed
$x_0\in X$, denote by $\pi_{x_0}:G\to X$ the smooth map $\pi_{x_0}(g)=g\cdot
x_0$. Suppose that:
\begin{enumerate}
\item
$\pi_{x_0}$ is an open mapping,  regarded as a map from $G$ onto the orbit
$\{g\cdot x_0: g\in G\}$ of $x_0$ (with the relative topology of $X$).
\item
The differential $d(\pi_{x_0})_1:(TG)_1\to X$ splits: its nullspace and range are
closed complemented subspaces in the Banach-Lie algebra $\mathcal{G}$ of $G$ and $X$, respectively.
\end{enumerate}
Then, the orbit $\{g\cdot x_0: g\in G\}$ is a smooth submanifold of  $X$, and the
map
$$
\pi_{x_0}:G\to \{g\cdot x_0: g\in G\}
$$ 
is a smooth submersion.
\end{lem}

\begin{prop}
Suppose that  $A^2-1$ has closed range. Then, $\d_{A_0}$ is a complemented $C^\infty$ submanifold of $\b(\h)\times\b(\h)$, and for any fixed $(P_0,Q_0)\in\d_{A_0}$, the map
$$
\pi_{(P_0,Q_0)}:\u_{\a_0}\to \d_{A_0} \ , \ \pi_{(P_0,Q_0)}(U)=(UP_0U^*,UQ_0U^*)
$$
 is a $C^\infty$ submersion.
\end{prop}
\begin{proof}
We shall apply Lemma \ref{raeburn} above. Note that the condition that $\pi_{(P_0,Q_0)}$ is open is fullfilled: if $A^2-1$ has closed range,  then $\pi_{(P_0,Q_0)}$ has continuous local cross-sections. A cross section on a neighbourhood of $(P_0,Q_0)$ was defined in Section 2  by
$$
{\bf s}:\{(P',Q')\in\d_{A_0}: P_0+Q'-1\in Gl(\h_0)\}\to \u_{\a_0} \ , \ {\bf s}(P',Q')=(P_0+Q'-1)|P_0+Q'-1|^{-1}V.
$$
This map can be extended to a map ${\bf \check{s}}$ defined on an open subset in $\b(\h_0)\times\b(\h_0)$, with values in $Gl(\h_0)$. Namely
$$
{\bf \check{s}}:\{(T,S)\in \b(\h_0)\times\b(\h_0): P_0+S-1\in Gl(\h_0)\} \ , \  {\bf \check{s}}(T,S)=(P_0+S-1)|P_0+S-1|^{-1}V.
$$  
Clearly, $\{(T,S)\in \b(\h_0)\times\b(\h_0): P_0+S-1\in Gl(\h_0)\}$ is an open subset of $\b(\h_0)\times\b(\h_0)$ containing  $(P_0,Q_0)$, and ${\bf \check{s}}$ is $C^\infty$.

The differential $d(\pi_{(P_0,Q_0)})_1: (\a_0)_{ah} \to \b(\h_0)\times \b(\h_0)$ is given by
$$
d(\pi_{(P_0,Q_0)})_1(Z)=(ZP_0-P_0Z,ZQ_0-Q_0Z).
$$
Here, $(\a_0)_{ah}$ denotes the set of anti-Hermitian elements of $\a_0$ (which is the Banach-Lie algebra of $\u_{\a_0}$).
Clearly, this map has a natural extension 
$$
\Pi:\b(\h_0)\to\b(\h_0)\times\b(\h_0) \ , \ \Pi(X)=(XP_0-P_0X,XQ_0-Q_0X).
$$
Denote  ${\bf \check{S}}=d({\bf \check{s}})_{(P_0,Q_0)}$. The fact that ${\bf s}$ is a cross section for $\pi_{(P_0,Q_0)}$ implies that $\pi_{(P_0,Q_0)}\circ {\bf s} \circ \pi_{(P_0,Q_0)}=\pi_{(P_0,Q_0)}$. Equivalently,
$$
 \pi_{(P_0,Q_0)}\circ {\bf \check{s}} \circ \pi_{(P_0,Q_0)}=\pi_{(P_0,Q_0)}.
$$
This is a composition of $C^\infty$ maps defined on open subsets of Banach spaces. If we differentiate this identity at $1$,  we get
\begin{equation}\label{rango}
\Pi\circ {\bf \check{S}}\circ \Pi=\Pi.
\end{equation}
If we  restrict this identity to $(\a_0)_{ah}$, the image $\Pi((\a_0)_{ah})$ equals the image of $d(\pi_{(P_0,Q_0)})_1$. Then, identity  (\ref{rango}) above  implies that $\Pi\circ{\bf \check{S}}$ is an idempotent whose range equals  the range of $\Pi$.
It follows that the range of   $d(\pi_{(P_0,Q_0)})_1$ is complemented in $\b(\h_0)\times\b(\h_0)$.

The nullspace of  $d(\pi_{(P_0,Q_0)})_1$ is 
$$
\{Z\in(\a_0)_{ah}: ZP_0=P_0Z \ \hbox{ and } \ ZQ_0=Q_0Z\}.
$$
This is the Banach-Lie algebra of $\pi_{(P_0,Q_0)}^{-1}(P_0,Q_0)$ (usually called the isotropy subgroup of the action at $(P_0,Q_0)$).  It was described in Remark \ref{isotropia} using Halmos representation. It is clear, then, that the Banach Lie algebra of the isotropy group consists of  matrices
$$
\left( \begin{array}{cc} X & 0 \\ 0 & X \end{array} \right)
$$
with $X^*=-X$ and  $X\Gamma=\Gamma X$.

Let us prove that this space is complemented in the Banach-Lie algebra $(\a_0)_{ah}$, which (in this representation) consists of matrices of the form  
$$
\left(\begin{array}{cc} X & Y \\ Y & Z \end{array} \right) 
$$
 where $X,Y,Z$ commute with $\Gamma$, and satisfy the equation
 $$
 C(X-Z)+2SY=0.
 $$

As remarked above,  $E_{(P_0,Q_0)}$ is a conditional expectation  between the von Neumann algebras $\a_0$ and the isotropy subalgebra $N(d(\pi_{(P_0,Q_0)})_1)$ at $(P_0,Q_0)$. Therefore, the anti-Hermitian part of the nullspace $N(E_{(P_0,Q_0)})\cap (\a_0)_{ah}$ is a supplement for the isotropy subalgebra.
\end{proof}

In the general case, i.e., if $R(A^2-1)$ is not necessarily closed, we shall use the transitive action of $\u_{\a_0}$ to induce a differentiable structure in  $\d_{A_0}$  . Using Halmos representation, we know the explicit form of the isotropy subgroups of the action.  In order to prove that $\d_{A_0}$ has a $C^\infty$   structure, and that the maps $\pi_{(P_0,Q_0)}$ are submersions, we shall use a general result on quotients of unitary groups (see, for instance, \cite{beltita}).
In this result, it is required that $H$ is a {\it Banach-Lie subgroup} of $G$ in the following  specific sense: 
\begin{defi} \label{def67} {\rm (\cite{beltita} Definition 4.1)}
 Let $G$ be a Banach-Lie group and $H$ a subgroup of $G$.
We say that $H$ is a Banach-Lie subgroup of $G$ if the following conditions are
satisfied.
\begin{enumerate}
\item
 The subgroup $H$ is endowed with a structure of Banach-Lie group whose
underlying topology is the same as  the relative topology of $H$ in $G$.
\item
The inclusion map $H \hookrightarrow G$ is smooth and the induced mapp between the Banach-Lie algebras  $L(H) \to L(G)$ is an injective operator with closed range.
\item
There exists a closed linear subspace $M$ of $L(G)$ such that $L(H)\oplus M = L(G)$.
\end{enumerate}
\end{defi}

\begin{prop} {\rm(\cite{beltita} Theorem  4.19)}
Let $G$ be a Banach-Lie group, $H$ a Banach-Lie subgroup
of $G$ and $\pi : G \to G/H$ the natural projection. Endow $G/H$ with the quotient
topology and consider the natural transitive action
$$
G \times G/H \to G/H, (g, kH)  \mapsto gkH.
$$
Then $G/H$ has a structure of $C^\infty$ manifold and  the following
conditions are satisfied:
\begin{enumerate}
\item
The mapping $\pi$ is $C^\infty$ and has $C^\infty$ local cross sections
near every point of $G/H$.
\item
 For every $g \in G$ the mapping
$$
 G/H \to G/H \ , \ \ kH\mapsto gkH
$$
is $C^\infty$.
\end{enumerate}
\end{prop}
The isotropy subgroup $\ii_{(P_0,Q_0)}:=\{W\in\u_{\a_0}: (WP_0W^*,WQ_0W^*)=(P_0,Q_0)\}\subset \u_{\a_0}$ clearly satisfies the conditions  of Definition \ref{def67} (the space  $M:=\mathbb{H}_{(P_0,Q_0)}$ given in  (\ref{reductivo}), satisfies condition 3.).
\begin{coro}
If $A\in\d$, the space $\d_{A_0}$ inherits a C$^\infty$ manifold structure from the quotient $\u_{\a_0} / \ii_{(P_0,Q_0)}$, which makes $\pi_{(P_0,Q_0)}$ a $C^\infty$ submersion.
\end{coro}
\begin{rem}
The topology that the quotient $\u_{\a_0} / \ii_{(P_0,Q_0)}$ induces in $\d_{A_0}$ might be different from the ambient topology induced by $\b(\h_0)\times \b(\h_0)$. In other words, the identification is not necessarily a homeomorphism between these topologies.
\end{rem}
\section{A reductive structure for $\d_{A_0}$}
Recall the supplement $\mathbb{H}_{(P_0,Q_0)}$ of $L(\ii_{(P_0,Q_0)})=N(E_{(P_0,Q_0)})\cap (\a_0)_{ah}$  in $(\a_0)_{ah}$, defined in \ref{el horizontal},
$$
\mathbb{H}_{(P_0,Q_0)}=\{ \left(\begin{array}{cc} -Y\tau & Y \\ Y & Y\tau\end{array}\right) : Y^*=-Y, Y\tau \hbox{ bounded}\}.
$$
This distribution of subspaces $\d_{A_0}\ni (P_0,Q_0)\mapsto \mathbb{H}_{(P_0,Q_0)}$   is what in differential geometry 
 is  called a {\it reductive structure} for $\d_{A_0}$, meaning that it satisfies the following conditions:
\begin{itemize}
\item
The subspace $\hh$ is invariant under the inner action of $\ii_{(P_0,Q_0)}$: if $W\in\ii_{(P_0,Q_0)}$ and  $Z\in\hh$
then $W\cdot Z\in\hh$.
\item
The  distribution of supplements $\d_{A_0}\ni (P',Q')\mapsto \mathbb{H}_{(P',Q')}$ is $C^\infty$.
This means that if ${\bf P}_{ \mathbb{H}_{(P',Q')}}$ denotes the idempotent (real) linear map acting in $(\a_0)_{ah}$ corresponding to the projection to the first component in the decomposition $(\a_0)_{ah}= \mathbb{H}_{(P',Q')}\oplus L(\ii_{(P',Q')})$, then the map
$$
\d_{A_0}\ni (P',Q')\mapsto {\bf P}_{ \mathbb{H}_{(P',Q')}}\in \b((\a_0)_{ah})
$$
is $C^\infty$.
\end{itemize}
 In our case, the fact that the supplement is the (anti-Hermitian part) of the nullspace of an $L(\ii_{(P_0,Q_0)})$-valued conditional expectation, implies the first property. Also note that  ${\bf P}_{ \mathbb{H}_{(P',Q')}}=\left(Id-E_{(P_0,Q_0)}\right)|_{ \b((\a_0)_{ah})}$.

Each pair $(P',Q')$ gives rise to a conditional expectation, which due to Remark \ref{independencia de Halmos}, depends only on the pair. We must show that  if the pairs $(P',Q')$ vary smoothly, then so do the maps $E_{(P',Q')}$.  The map  $B\mapsto V$ , via the polar decomposition $B=V|B|$,  in general is not continuous, much less smooth. However, in the unitary orbit of $(P_0,Q_0)$, it is smooth. Indeed, note that, locally (for $(P',Q')$ close to $(P_0,Q_0)$), the unitary $U$ in $\a_0$ such that $(UP_0U^*,UQ_0U^*)=(P',Q')$  can be chosen as a smooth map in the arguments $P',Q'$, by means of smooth local cross section for the submersion $\pi_{(P_0,Q_0)}$. Then, if we denote (as in Remark \ref{independencia de Halmos}), $B=Q_0-P_0Q_0P_0+P_0^\perp Q_0 P_0^\perp$, and, accordingly, $B'=Q'-P'Q'P'-P'^\perp Q' P'^\perp$, and $V$ and $V'$ are the isometric parts in the corresponding polar decompositions of $B$ and $B'$, then
$$
UVU^*=V'.
$$
Therefore, $E_{(P',Q')}=U E_{(P_0,Q_0)}(U^*\cdot U) U^*$.

\begin{rem}
Elements $Z\in\hh$ have  symmetric spectrum, with symmetric multiplicity. Indeed, consider 
$$
J=\left(\begin{array}{cc} 0 & 1 \\ -1 & 0 \end{array}\right).
$$
Then $J^*=-J$, $J^2=-1$ (a fortiori, $J$ is a unitary in $\h_0$), and for any $Z\in\hh$,
$$
J\left(\begin{array}{cc} -Y\tau & Y \\ Y & Y\tau \end{array} \right)=-\left(\begin{array}{cc} -Y\tau & Y \\ Y & Y\tau \end{array} \right)J,
$$
i.e., $JZJ^*=-ZJJ^*=-Z$. 

Therefore, if one considers the symmetry $V=iJ$, it turns out that elements in $\mathbb{H}_{(P_0,Q_0)}$ are of the form $Z=r(E-F)$,  for $r=\|Z\|$ and  $E,F$ orthogonal projections.  
\end{rem}

\begin{rem}
A reductive structure on a homogeneous space induces a linear connection (see \cite{kn}, or \cite{m-l r} for an infinite dimensional setting). For instance, the geodesics can be explicitly computed (in terms of $YC^{-1}$). Indeed,  note that,  if $Z=\left( \begin{array}{cc} -Y\tau & Y \\ Y & Y\tau \end{array} \right)$, then
$$
Z=Z_0+Z_1 ,  \ \hbox{ where } Z_0=\left( \begin{array}{cc} -Y\tau & 0 \\ 0 & Y\tau \end{array} \right) \ \hbox{ and } \ Z_1=\left( \begin{array}{cc} 0 & Y \\ Y & 0 \end{array} \right),
$$ 
and $Z_0, Z_1$ anti-commute. Thus, 
$$
Z^2=Z_0^2+Z_1^2=\left( \begin{array}{cc} Y^2\tau^2+Y^2 & 0 \\ 0 & Y^2\tau^2+Y^2 \end{array} \right)=(YC^{-1})^2\left( \begin{array}{cc} 1 & 0 \\ 0 & 1  \end{array} \right),
$$
and $Z^{2n}=(YC^{-1})^{2n}\left( \begin{array}{cc} 1 & 0 \\ 0 & 1 \end{array} \right)$.
Also 
$$
Z^{2n+1}=Z^{2n}Z=\left( \begin{array}{cc} -(YC^{-1})^{2n}Y\tau & (YC^{-1})^{2n}Y \\ (YC^{-1})^{2n}Y & (YC^{-1})^{2n}Y\tau \end{array} \right)=\left( \begin{array}{cc} -(YC^{-1})^{2n+1}S & (YC^{-1})^{2n+1}YC \\ (YC^{-1})^{2n+1}YC & (YC^{-1})^{2n+1}YS \end{array} \right)
$$
$$
=(YC^{-1})^{2n+1} \left( \begin{array}{cc} -S & C \\ C & S \end{array} \right).
$$
Then, 
\begin{equation}\label{etZ1}
e^{tZ}=\cosh(tYC^{-1}) \left( \begin{array}{cc} 1 & 0 \\ 0 & 1 \end{array} \right)+\sinh(tYC^{-1}) \left( \begin{array}{cc} -S & C \\ C & S \end{array} \right).
\end{equation}
The operator $\Sigma=\left( \begin{array}{cc} -S & C \\ C & S \end{array} \right)$ is a symmetry: $\Sigma^*=\Sigma=\Sigma^{-1}$.

Note that, since $YC^{-1}$ is anti-Hermitian,  and $\cosh$ and $\sinh$ are, respectively, even and odd functions, then
 the first term of  $e^{tZ}$ is selfadjoint and the second is anti-Hermitian. Then,
$$
e^{-tZ}=\cosh(tYC^{-1}) \left( \begin{array}{cc} 1 & 0 \\ 0 & 1 \end{array} \right)-\sinh(tYC^{-1}) \left( \begin{array}{cc} -S & C \\ C & S \end{array} \right).
$$
Therefore, the geodesic $\delta(t)=(e^{tZ}P_0e^{-tZ}, e^{tZ}Q_0e^{-tZ})$ can be explicitely computed.

Alternatively, let $D^*=D$ such that $YC^{-1}=iD$. Then, $\cosh(tYC^{-1})=\cos(tD)$ and $\sinh(tYC^{-1})=i \sin(tD)$.
\end{rem}
\begin{rem}
The geodesics can be described in an intrinsic way, without reference to the Halmos frame of reference. To this effect, note that the matrix $\left( \begin{array}{cc} -S & C \\ C & S \end{array} \right)$ above is precisely $-J_0$, ($J_0=sgn(A_0)$). Indeed:  
$$
A_0=\left( \begin{array}{cc} S^2 & -CS \\ -CS & -S^2 \end{array} \right) \ , \ 
A_0^2=\left( \begin{array}{cc} S^2 & 0 \\ 0  & S^2 \end{array} \right) \ \hbox{ and  } |A_0|=\left( \begin{array}{cc} S & 0 \\ 0  & S \end{array} \right).
$$
Next, note that
$$
P_0ZP_0J_0P_0=\left( \begin{array}{cc} -Y\tau & 0 \\ 0  & 0 \end{array} \right)\left( \begin{array}{cc} -S & 0 \\ 0  & 0 \end{array} \right)=\left( \begin{array}{cc} YC^{-1} & 0 \\ 0  & 0 \end{array} \right).
$$
Then
\begin{equation} \label{etZ2}
e^{tZ}=\cosh(t P_0ZP_0J_0P_0)+\sinh(t P_0ZP_0J_0) J_0,
\end{equation}
and similarly for $e^{-tZ}$.
\end{rem}

With these expressions above, the bijectivity radius of the exponential map can be computed.
\begin{teo}\label{radio geodesico}
Let $(P_0,Q_0)\in\d_{A_0}$. Then the exponential map
$$
exp_{(P_0,Q_0)}:\{Z\in\mathbb{H}_{(P_0,Q_0)}: \|Z\|<\pi/2\} \to exp_{(P_0,Q_0)}\left(\{Z\in\mathbb{H}_{(P_0,Q_0)}: \|Z\|<\pi/2\}\right)
$$ 
is a bijection, whose image
$$
exp_{(P_0,Q_0)}\left(\{Z\in\mathbb{H}_{(P_0,Q_0)}: \|Z\|<\pi/2\}\right)\subset \d_{A_0}
$$
is an open dense subset.
\end{teo}
\begin{proof}
 Let $v_1,v_2\in(T\d_{A_0})_{(P_0,Q_0)}$ with the same exponential, i.e.,   if  $Z_1,Z_2\in\hh$ are the corresponding horizontal elements, $e^{Z_1}\cdot (P_0,Q_0)=e^{Z_2}\cdot (P_0,Q_0)$. Then, following the  notations of the preceeding remark ($Z_j=\left(\begin{array}{cc} -Y_j\tau & Y_j \\ Y_j & Y_j\tau \end{array} \right)$ and $D_j=-iY_j\tau$, for $j=1,2$), there exists a unitary operator in $\l$ such that 
$$
e^{-Z_2}e^{Z_1}=\left( \begin{array}{cc} W & 0 \\ 0 & W \end{array} \right),
$$
i.e., $\sin(D_2)W=\sin(D_1)$ and $\cos(D_2)W=\cos(D_1)$. If we suppose that $|v_j|_{(P_0,Q_0)}=\|D_j\|<\pi/2$, then the cosines are invertible, and, then, these identities  imply that  $D_1=D_2$. 

Let us prove that  $exp_{(P_0,Q_0)}\left(\{Z\in\mathbb{H}_{(P_0,Q_0)}: \|Z\|<\pi/2\}\right)$ is an open dense subset of $\d_{A_0}$. In order to prove this fact, we shall use the alternative characterization for $\d_{A_0}$ given in Theorem \ref{davis}. Namely, representing the elements of $\d_{A_0}$ as symmetries $V$ which anti-commute with $A_0$ . Recall that the correspondence is given by $(P,Q)\leftrightarrow V$, where $V$ is the symmetry (the isometric part) in the polar decomposition of $P+Q-1$.   Let us compute $V_0$, the symmetry corresponding to the pair $(P_0,Q_0)$,  using Halmos decomposition based on the pair $(P_0,Q_0)$. Note that
$$
P_0+Q_0-1=\left( \begin{array}{cc} C^2 & CS \\ CS & -C^2 \end{array}\right) \  , \ \  (P_0+Q_0-1)^2= \left( \begin{array}{cc} C^2 & 0\\ 0 & C^2 \end{array}\right),
$$
so that 
$$
|P_0+Q_0-1|=\left( \begin{array}{cc} C & 0 \\ 0 & C \end{array}\right) \ \hbox{ and } \ \ V_0=\left( \begin{array}{cc} C & S \\ S & -C \end{array}\right) .
$$
Pick $Z\in\mathbb{H}_{(P_0,Q_0)}$, $Z=\left( \begin{array}{cc} -Y\tau & Y \\ Y & Y\tau \end{array}\right)$.
Recall that $Y$ and $Y\tau$ commute with $C$ and $S$, so we get
$$
ZV_0=\left( \begin{array}{cc} -Y\tau & Y \\ Y & Y\tau \end{array}\right)\left( \begin{array}{cc} C & S \\ S & -C \end{array}\right)=\left( \begin{array}{cc} 0 & -Y\tau S-YC \\ YC+Y\tau S & 0 \end{array}\right)
$$
and
$$
V_0Z=\left( \begin{array}{cc} 0 & Y\tau S+YC \\ -YC-Y\tau S & 0 \end{array}\right).
$$
That is, $ZV_0=-V_0Z$. Conversely, a similar computation shows that anti-Hermitian elements $Z$ which anti-commute with $V_0$ belong to $\mathbb{H}_{(P_0,Q_0)}$. Pick a symmetry $V$ which anti-commutes with $A_0$, such that $\|V-V_0\|<2$. In general, any pair of symmetries lie at distance less or equal than $2$. Therefore, the set of such $V$ form an open dense subset of $\d_{A_0}$. We claim that these elements $V$ belong to the image of the exponential $exp_{(P_0,Q_0)}$ restricted to $\{Z\in\mathbb{H}_{(P_0,Q_0)}: \|Z\|<\pi/2\}$. In \cite{pr}, H. Porta and L.Recht proved that a symmetry $V$ such that $\|V-V_0\|<2$ is of the form $V=e^{Z}V_0e^{-Z}$, where $Z$ is an anti-Hermitian operator which anti-commutes with $V_0$, i.e.,  $Z\in \mathbb{H}_{P_0,Q_0}$, with $\|Z\|<\pi/2$. 
Also note that,  since $Z$ anti-commutes with $V_0$, $V_0e^{-Z}=e^{Z}V_0$, thus,
$$
V=e^{2Z}V_0 \ \hbox{ and  } \ e^{2Z}=VV_0.
$$
Since $\|2Z\|<\pi$, $Z$ can be obtained as $Z=\frac12 \log(VV_0)$ ($\log$ the unique logarithm for unitary operators $U$ such that $\|U-1\|<2$; notice that $\|VV_0-1\|=\|V-V_0\|<2$). Since $V$ and $V_0$ anti-commute with $A_0$, $VV_0$ and its logarithm  $Z$ belong to $\a_0$. 

The correspondence between symmetries and pairs in $\d_{A_0}$ is clearly equivariant: $UV_0U^*\leftrightarrow U\cdot (P_0,Q_0)$, for $U\in \u_{\a_0}$. Therefore, $(P,Q)=exp_{(P_0,Q_0)}(Z)$, where $Z\in \{Z\in\mathbb{H}_{(P_0,Q_0)}: \|Z\|<\pi/2\}$.  
\end{proof}

A similar argument (representing $\d_{A_0}$ as Davis' symmetries) allows one to prove that \\ $exp_{(P_0,Q_0)}$ is globally onto. The proof is also a refinement of the argument which showed that the action of $\u_{\a_0}$ on $\d_{A_0}$ is transitive. In fact, the result below implies the former.

\begin{teo}\label{expo onto}
Let $(P_0,Q_0)\in\d_{A_0}$. The exponential map
$$
exp_{(P_0,Q_0)}:\{Z\in\mathbb{H}_{(P_0,Q_0)}: \|Z\|\le \pi/2\}\to \d_{A_0}
$$
is onto.
\end{teo}
\begin{proof}
Pick $(P,Q)\in\d_{A_0}$, and let $V=sgn(P+Q-1)$ be its Davis' symmetry. Let $V_0$ be the symmetry corresponding to $(P_0,Q_0)$. Denote by 
$$
\h_{+,+}=\{\xi\in\h: V_0\xi=\xi \hbox{ and } V\xi=\xi\}  \ ,   \  \h_{-,-}=\{\xi\in\h: V_0\xi=-\xi \hbox{ and } V\xi=-\xi\} \ , 
$$
$$
\ \h_{+,-}=\{\xi\in\h: V_0\xi=\xi \hbox{ and } V\xi=-\xi\} \  \hbox{ and } \ \h_{-,+}=\{\xi\in\h: V_0\xi=-\xi \hbox{ and } V\xi=\xi\} . 
$$
Recall  the symmetry $J_0$ obtained as the sign of $A_0$, which anti-commutes with $V$ and $V_0$.
Note that, if $V_0\xi=\xi$, then $J_0\xi\ne 0$ satisfies $V_0J_0\xi=-J_0V_0\xi=-J_0\xi$, and similarly for $\xi$ such that $V_0\xi=-\xi$. The analogous property holds for $V$. Therefore, $J_0$ maps $\h_{+,+}$ into $\h_{-,-}$ and viceversa, and it is one to one between these subspaces. The same happens for $\h_{+,-}$ and $\h_{-,+}$.
In \cite{pemenoscu}, it was proven that there exists a geodesic of the Grassmann manifold of $\h$ (=space of symmetries of $\h$) joining $V$ and $V_0$ if and only if 
$$
\dim(\h_{+,-})=\dim(\h_{-,+}).
$$
It follows that there exists $Z$, $Z^*=-Z$, $\|Z\|\le\pi/2$, which is co-diagonal with respecto to $V_0$, such that $e^{Z}V_0e^{-Z}=V$. Such $Z$ may not be unique (if the dimensions above are non zero). To finish the proof, we must show that we can find one of these $Z$ in $\a_0$. In \cite{pemenoscu} it was shown that all such $Z$ are reduced by the decomposition
$$
\h_0=\h_0'\oplus \h_0''\oplus\h_0^0,
$$
where $\h_0'=\h_{+,+}\oplus\h_{-,-}$, $\h_0'=\h_{+,-}\oplus\h_{-,+}$ and $\h_0^0$ is the orthogonal complement to the sum of these two.   Both $J_0$ and $A_0$ are reduced by this decomposition. Clearly, also $V$ and $V_0$ are also reduced.

$Z$ is trivial in $\h_0'$ (\cite{pemenoscu}). In $\h_0''$, we choose $Z|_{\h_0''}=i\pi/2J_0|_{h_0''}$. The exponential of this anti-Hermitian operator yields a unitary operator intertwinig $V_0|_{\h_0''}$ and $V|_{\h_0''}$. In the third subspace, which,  in fact, is the generic part of the pair $V_0,V$, the exponent $Z|_{\h_0^0}$ is obtained uniquely as the logarithm (see \cite{pemenoscu})
$$
Z|_{\h_0^0}=\log({\bf S}V_0|_{\h_0^0}),
$$
where ${\bf S}$ is the Davis symmetry of the reductions $V_0|_{\h_0^0}$ and $V|_{\h_0^0}$  to their generic part: 
$$
{\bf S}=sgn(-\frac12\{V_0+V\}|_{\h_0^0}).
$$
$\{V+V_0\}|_{\h_0^0}$ anti-commutes with  $A_0|_{\h_0^0}$, then so does its unitary part ${\bf S}$ in the polar decomposition. It follows that ${\bf S}V_0$ commutes with $A_0|_{\h_0^0}$. Therefore, our choice of $Z$ belongs to $\a_0$ and the proof is complete .
\end{proof}

\section{A Hopf-Rinow theorem for $\d_{A_0}$}
In \cite{coco}, Dur\'an, Mata-Lorenzo and Recht introduced a Finsler metric for homogeneous spaces $\u_\a/\u_\b$ which are obtained as the quotient of the unitary group of a C$^*$-algebra by the unitary group of a C$^*$-subalgebra $\b\subset\a$. Denote by
$$
\pi:\u_\a\to \u_\a/\u_\b
$$
the quotient map. A tangent vector $v$ at $[1]$ (the class of $1$ in $\u_\a/\u_\b$) identifies with an element in the quotient of (real) Banach spaces $\a_{ah} /\b_{ah}$ ($\a_{ah}$ and $\b_{ah}$ denote the spaces of anti-Hermitian elements of $\a$ and $\b$, respectively). Thus $v=d(\pi)_1(z)$, for some $z\in\a_{ah}$; we shall say that $z$ is a {\it lifting} of $v$. The Finsler norm $|v|_{[1]}$ is defined as the infimum
$$
|v|=|v|_{[1]}=\inf\{\|z\|: z\in\a_{ah} \hbox{ is a lifting of } v\}.
$$
Or equivalently, if $z_0$ is an arbitrary lifting of $v$, 
$$
|v|=\inf\{\|z_0+y\|: y\in\b_{ah}\}.
$$
The metric is carried over the entire tangent bundle by means of the left action of $\u_\a$ on $\u_\a/\u_\b$. Thus,  the metric so defined is invariant under this action.  A lifting $z_0$ which achieves this infimum is called a {\it minimal lifting}. Minimal liftings may not exist (see for instance \cite{varela} for a nice example), and may fail to be unique (see \cite{matrices} for a finite dimensional example), but if  $\b\subset\a$ are von Neumann algebras, they do exist \cite{coco2}. We shall apply this theory to $\d_{A_0}$, where the algebras are indeed von Neumann algebras (being defined as commutant of selfadjoint  operators). The main result in \cite{coco} states that if $z_0\in\a_{ah}$ is a minimal lifting of $v$, then
$$
\delta(t)=[e^{tz}]
$$
is a curve of minimal length for time $|t|\le \frac{\pi}{2|v|}$.

In our context   $L(\ii_{(P_0,Q_0)})$ is the space of anti-Hermitian elements of the von Neumann algebra of operators which commute with $P_0$ and $Q_0$, inside the von Neumann algebra of operators which commute with $A_0=P_0-Q_0$. The fact that $\hh$ is a supplement for $L(\ii_{(P_0,Q_0)})$ in $(\a_0)_{ah}$, implies that any tangent vector $v$ in $T(\d_{A_0})_{(P_0,Q_0)}$ has a lifting $Z\in\hh$. We claim that  this lifting is minimal.
To prove this we need the following result:
\begin{prop}\label{prop616} {\rm (Proposition 5.2 of \cite{coco}, see also Theorem 2.2 of \cite{matrices})} 
Let $\b\subset\a$ be C$^*$-algebras, and let $Z\in\a_{ah}$. Suppose that there exists a state $\psi$ in $\a$  with the following properties:
\begin{enumerate}
\item
$\psi(Z^2)=-\|Z\|^2$.
\item
For any $Y\in\b$, $\psi(YZ)=0$.
\end{enumerate}
Then $Z$ is a minimal lifting (i.e., $\|Z\|\le \|Z+D\|$, for all $D\in\b_{ah}$).
\end{prop}
\begin{lem}\label{lema617}
Let $Z^*=-Z$, with matrix
$$
Z=\left( \begin{array}{cr} X & Y \\ Y & -X \end{array} \right)
$$
 such that $X$ and $Y$ commute. Then 
$$
\|Z\|\le \|Z+D\|,
$$
for any $D^*=-D$ of the form $D=\left(\begin{array}{cc} D' & 0 \\ 0 & D' \end{array} \right)$.
\end{lem} 
\begin{proof}
We use Proposition \ref{prop616} above.
Note that  
$$
Z^2=\left( \begin{array}{cc} X^2+Y^2 & XY-YX \\  YX-XY &  X^2+Y^2 \end{array} \right)=\left( \begin{array}{cc} X^2+Y^2 & 0 \\ 0 & X^2+Y^2 \end{array} \right) .
$$
Note, also,  that $X^2+Y^2\le 0$, and that $\|Z^2\|=\|X^2+Y^2\|$. There exists a state $\psi_0$ in $\b(\l)$ such that $\psi_0(X^2+Y^2)=-\|X^2+Y^2\|$. Let $\tau$ the positive unital linear map
$$
\tau(\left( \begin{array}{cc} T_{11} & T_{12} \\ T_{21} & T_{22} \end{array} \right) )=\frac12\{T_{11} + T_{22}\}.
$$
Then $\psi=\psi_0\circ \tau$ is a state in $\b(\h)$ such that
\begin{enumerate}
\item
$$
\psi(Z^2)=\psi_0(X^2+Y^2)=-\|X^2+Y^2\|.
$$
\item
For $D^*=-D$ of the form, $D=\left( \begin{array}{cc} D' & 0 \\ 0 & D' \end{array} \right)$,
$$
\tau(ZB)=\tau ( \left( \begin{array}{cc} XD' & YD' \\ YD' & -XD' \end{array} \right)=0.
$$
Then $\psi(ZB)=0$.
\end{enumerate}
\end{proof}

If $v\in T(\d_{\a_0})_{(P_0,Q_0)}$, we shall denote by $Z_v\in\hh$ the unique lifting of $v$ which belongs to $\hh$. Note that the mapping $v\mapsto Z_v$ is a linear isomorphism (it is the inverse of $d(\pi_{(P_0,Q_0)})_1$ restricted to $\hh$).

The following result follows:

\begin{teo}\label{geodesica minimal}
Let $(P_0,Q_0)\in\d_{A_0}$ and $v\in T(\d_{A_0})_{(P_0,Q_0)}$ such that $Z_v=\left( \begin{array}{cc} Y\tau & Y \\ Y & -Y\tau \end{array} \right)$.
Then 
$$
\delta(t)=e^{tZ_v}\cdot (P_0,Q_0)=(e^{tZ_v}P_0e^{-tZ_v},e^{tZ_v}Q_0e^{-tZ_v})
$$
is a minimal geodesic for the Finsler metric defined above, up to time 
$$
|t|\le \frac{\pi}{2|v|_{(P_0,Q_0)}}=\frac{\pi}{2\|Z_v\|}=\frac{\pi}{2\|Y^2+(Y\tau)^2\|^{1/2}}.
$$
\end{teo}
\begin{proof}
Recall from Remark \ref{CONMUTAN}, that for any $Z_v\in\mathbb{H}_{(P_0,Q_0)}$, all entries in the matrix commute. Thus,  Lemma \ref{lema617}  holds and $Z_v$ has minimal norm among all perturbations with elements in the isotropy algebra.
\end{proof}

The fact that the minimal liftings belong to the linear space $\hh$ implies,  additionally, the following result, which can be regarded as a Hopf-Rinow Theorem for the space $\d_{A_0}$ (in a context which is far  from being Riemannian). Note that the pairs in $\d_{A_0}$ which can be reached from any given pair $(P_0,Q_0)$ by a {\it unique} minimal geodesic, is an open dense subset of $\d_{A_0}$. Theorem \ref{radio geodesico} and Theorem \ref{geodesica minimal} imply
\begin{coro}
Let  $(P,Q), (P_0,Q_0)\in\d_{A_0}$, and let $V$ and $V_0$ be the corresponding Davis' symmetries: $V=sgn(P+Q-1)$, $V_0=sgn(P_0+Q_0-1)$.   If  $\|V-V_0\|<2$, then there exists a unique element $Z=Z_{(P_0,Q_0)}(P,Q)$, which is a $C^\infty$ map in terms of the arguments $(P_0,Q_0), (P,Q)$, such that
$$
(P,Q)=e^{Z}\cdot (P_0,Q_0).
$$
The geodesic $\delta(t)=e^{tZ}\cdot (P_0,Q_0)$ has minimal length among all piecewise smooth curves in $\d_{A_0}$ joining $(P_0,Q_0)$ with $(P,Q)$.
\end{coro}
\begin{proof}
From the proof of of Theorem \ref{radio geodesico}, it is clear that $\|Z\|<\pi/2$. Then, by Theorem \ref{geodesica minimal}, $\delta$ is minimal up to time $t=1$, and $\delta(1)=(P,Q)$.
\end{proof}
Using now Theorem \ref{expo onto}, one can show that (dropping the uniqueness condition), any pair of elements in $\d_{A_0}$ can be joined by a minimal geodesic.
\begin{coro}
Let $(P_0,Q_0), (P,Q)\in\d_{A_0}$. Then there exists a minimal geodesic of $\d_{A_0}$, of length less or equal than $\pi/2$, which joins them.
\end{coro}
\begin{proof}
The existence of a geodesic $\delta$, parametrized in the interval $[0,1]$,  joining $(P_0,Q_0)$ and $(P,Q)$ is a direct consequence of Theorem \ref{expo onto}. Also, it is clear there that the norm of the exponent $Z$ (=the length of the geodesic) is less or equal than $\pi/2$. Let us prove, by a standard geometric argument, that it must be minimal. Clearly, by the above result, we must consider the case when the length $\ell(\delta)$ of $\delta$ is $\pi/2$. Suppose there is a curve $\gamma\in\d_{A_0}$ joining $(P_0,Q_0)$ of length $\pi/2-r$, for some $r>0$. Pick $t_0$ such that $(1-t_0)\pi/2=\ell(\delta|_{[t_0,1]})<r/2$. Then the length of the curve obtained by adjoining to $\gamma$ the curve $\delta|_{[t_0,1]}$ reversed,  gives a curve joining $(P_0,Q_0)$ and $\delta(t_0)$, of length strictly less than $t_0\pi/2$. This is a contradiction, since by the above Corollary, the curve $\delta|_{[0,t_0 ]}$ has minimal length $t_0\pi/2$.
\end{proof}

\begin{rem}
Note that if $Z=\left( \begin{array}{cc} -Y\tau & Y \\ Y & Y\tau \end{array} \right)$ is a horizontal element, then also $YC^{-1}$ is bounded. Recall that $X,Y,Z$ commute with $C,S,\tau$, the equality $C(Z-X)=2SY$ implies that 
$$
C^2(Z-X)^2 4S^2Y^2=4Y^2-4C^2Y^2 \ , \ \hbox{ i.e., } \ \ C^2 A=Y^2,
$$
where $A=\frac14((X-Z)^2+Y^2)$ is bounded. It follows that $Y^2C^{-2}$ is bounded, thus its square root $|Y|c^{-1}$ is bounded, and then also $YC^{-1}$ is bounded.
Moreover, if $Z=\left(\begin{array}{cc} -Y\tau & Y \\ Y & Y\tau \end{array} \right)$, then 
$$
\|Z\|=\|Y^2\tau^2+Y^2\|^{1/2}=\|YC^{-1}\|.
$$
Therefore, if $\|Z\|\le \pi/2$,  the distance between $(P_0,Q_0)$ and $e^{Z}\cdot(P_0,Q_0)$ equals $\|YC^{-1}\|$.
\end{rem}

Supported by PIP 0525 (CONICET).

\bigskip

Esteban Andruchow \\
Instituto de Ciencias,  Universidad Nacional de Gral. Sarmiento,
\\
J.M. Gutierrez 1150,  (1613) Los Polvorines, Argentina
\\ 
and Instituto Argentino de Matem\'atica, `Alberto P. Calder\'on', CONICET, 
\\
Saavedra 15 3er. piso,
(1083) Buenos Aires, Argentina.
\\
e-mail: eandruch@ungs.edu.ar

\bigskip

Gustavo Corach\\
Instituto Argentino de Matem\'atica, `Alberto P. Calder\'on', CONICET,
\\
Saavedra 15 3er. piso, (1083) Buenos Aires, Argentina,
\\
and Depto. de Matem\'atica, Facultad de Ingenier\'\i a, Universidad de Buenos Aires, Argentina.
\\
e-mail: gcorach@fi.uba.ar

\bigskip

L\'azaro Recht \\
Departamento de Matem\'atica P y A, 
Universidad Sim\'on Bol\'\i var \\
Apartado 89000, Caracas 1080A, Venezuela  \\
e-mail: recht@usb.ve

\end{document}